\documentclass[11pt]{article}

\usepackage[backend=bibtex, maxbibnames=9,maxcitenames=2, doi=false, url=false, isbn=false]{biblatex}
\DeclareNameAlias{sortname}{last-first}
\usepackage{fullpage}
\usepackage{avant}
\usepackage[T1]{fontenc}
\usepackage{hyperref, pdfsync}
\usepackage{xcolor}
\definecolor{link-color}{rgb}{0.15,0.4,0.15}
\hypersetup{
  colorlinks,
  linkcolor = {link-color}, citecolor = {link-color},
}

\usepackage{graphicx}
\usepackage[font=small,labelfont=bf]{caption}
\usepackage{subcaption}
\usepackage{hyperref}
\usepackage{amsmath,amsthm,amssymb,enumitem}
\usepackage{bbm}
\usepackage{bm}
\usepackage{tabularx}

\newtheorem{theorem}{Theorem}[section]
\newtheorem{prop}[theorem]{Proposition}
\newtheorem{cor}[theorem]{Corollary}
\newtheorem{lemma}[theorem]{Lemma}
\theoremstyle{plain}
\newtheorem{rem}[theorem]{Remark}
\newtheorem{defn}[theorem]{Definition}

\newcommand{\R}{\mathbb{R}}

\renewcommand{\P}{\mathbb{P}}

\newcommand{\1}{\mathbbm{1}}

\renewcommand{\d}{\,d}

\newcommand{\bP}{\pmb{\rm P}}
\newcommand{\bE}{\pmb{\rm E}}

\bibliography{references.bib}

\title{Deep factorisation of the stable process II: \\
potentials and applications}

\author{Andreas E. Kyprianou, Victor Rivero and Bat\i{} \c{S}eng\"ul}

\begin{document}

  \maketitle

  \begin{abstract}
    Here, we propose a different perspective of the {\it deep factorisation} in \cite{kyprianou2015deep} based on determining potentials. Indeed, we factorise the inverse of the MAP--exponent associated to a stable process via the Lamperti--Kiu transform.  Here our factorisation is completely independent from the derivation in \cite{kyprianou2015deep}, moreover there is no clear way to invert the factors in \cite{kyprianou2015deep} to derive our results. Our method gives direct access to the potential densities of the ascending and descending ladder MAP of the Lamperti-stable MAP in closed form.

    In the spirit of the  interplay between the classical Wiener--Hopf factorisation and the fluctuation theory of the underlying L\'evy process, our analysis will produce a collection of  new results for stable processes. We give an identity for the law of the point of closest reach to the origin for a stable process with index $\alpha\in(0,1)$ as well as an identity for the the law of the point of furthest reach before absorption at the origin for a stable process with index $\alpha\in(1,2)$. Moreover, we show how the deep factorisation allows us to compute explicitly the limiting distribution of stable processes multiplicatively reflected in such a way that it remains in the strip $[-1,1]$.
    \medskip

    \noindent{\it Subject classification:} 60G18,  	60G52, 60G51. \\

    \noindent{\it Key words:} Stable processes, self-similar Markov processes, Wiener--Hopf factorisation, radial reflection.

    %
  \end{abstract}
  \section{Introduction and main results}

 Let $(X,\P_x)$ be a one-dimensional L\'evy process started at $x\in\mathbb{R}$. Suppose that, when it exists, we write $\psi$ for its Laplace exponent, that is to say, $\psi(z): = t^{-1}\log \mathbb{E}_0[e^{z X_t}]$ for all $z\in\mathbb{C}$ such that the right-hand side exists.
An interesting aspect of the characteristic exponent of L\'evy processes is that they can be written as a product of the so called Wiener--Hopf factors, see for example \cite[Theorem 6.15]{MR3155252}. This means that there exists two Bernstein functions $\kappa$ and $\hat\kappa$ (see \cite{MR2978140} for a definition) such that, up to a multiplicative constant,
\begin{equation}\label{eq:normal_WH}
  -\psi(i\theta)=\kappa(-i\theta)\hat\kappa(i\theta), \qquad \theta\in\mathbb{R}.
\end{equation}
There are, of course, many functions $\kappa$ and $\hat\kappa$ which satisfy \eqref{eq:normal_WH}. However imposing the requirements that the functions $\kappa$ and $\hat\kappa$ must be Bernstein functions which analytically extend in $\theta$ to the upper and lower half plane of $\mathbb{C}$, respectively, means that the factorisation in \eqref{eq:normal_WH} is unique up to a multiplicative constant.

The problem of finding such factors has considerable interest since, probablistically speaking, the Bernstein functions $\kappa$ and $\hat\kappa$ are the Laplace exponents of the ascending and descending ladder height processes respectively. The ascending and descending ladder height processes, say $(h_t: t\geq 0)$ and $(\hat{h}_t: t\geq 0)$, are subordinators that correspond to a time change of $\overline{X}_t: = \sup_{s\leq t}X_s$, $t\geq 0$ and $-\underline{X}_t: = -\inf_{s\leq t}X_s$, $t\geq 0$, and therefore have the same range, respectively. Additional information comes from the exponents in that they also provide information about the potential measures associated to their respective ladder height processes. So for example, $U(dx): = \int_0^\infty\mathbb{P}(h_t\in dx)dt$, $x\geq 0$, has Laplace transform given by $1/\kappa$. A similar identity hold for $\widehat{U}$, the potential of $\hat{h}$.

These potential measures appear in a wide variety of fluctuation identities. Perhaps the most classical example concerns the stationary distribution of the process reflected in its maximum, $\overline{X}_t - X_t$, $t\geq 0$ in the case that $\lim_{t\to\infty}X_t = \infty$. In that case, we may take $\lim_{t\to\infty}\mathbb{P}_x(\overline{X}_t -X_t\in dx) = \kappa(0)\widehat{U}(dx)$; c.f. \cite{MR0217858}.
The ladder height potential measures also feature heavily in first passage identities that describe the joint law of the triple $(X_{\tau^+_x}, X_{\tau^+_x-}, \overline{X}_{\tau^+_x-})$, where $\tau^+_x = \inf\{t>0: X_t >x\}$ and $x>0$; cf. \cite{DK}.
Specifically, one has,
for  $u>0$, $v\geq y$ and $y\in[0,x]$,
\[ \mathbb{P}(X_{\tau^+_x } -
    x\in{{d}}u, x - X_{\tau^+_x -} \in {{d}}v , x -
    \overline{X}_{\tau^+_x
  -}\in{{d}}y)
  = U( x-dy)\widehat{U}({{d}}v-y)\Pi({{d}}u + v),
\]
where $\Pi$ is the L\'{e}vy measure of $X$.
A third example we give here pertains to the much more  subtle property of increase times. Specifically, an increase time is a (random) time $t>0$ at which there exists a (random) $\varepsilon>0$ such that $X_{t'}\leq X_t\leq X_{t''}$ for all $t'\in[t-\varepsilon, t]$ and $t''\in[t,t+\varepsilon]$.  The existence of increase times occurs with probability 0 or 1. It is known that under mild conditions, increase times exist almost surely if and only if  $\int_{0+}\widehat{U}(x)^{-1}U(dx)<0$. See, for example, \cite{Fourati} and references therein. 

Within the class of L\'evy processes which experience jumps of both sign, historically there have been very few  explicit cases of the Wiener--Hopf factorisation identified within the literature. More recently, however,  many new cases have emerged hand-in-hand  with new characterisations and distributional identities of path functionals of stable processes; see e.g. the summary in \cite[Section 6.5 and Chapter 13]{MR3155252}.  A L\'evy process $(X,\P_x)$ is called a (strictly) $\alpha$-stable process for $\alpha\in(1,2]$ if for every $c>0$, $(cX_{c^{-\alpha}t}:t \geq 0)$ under $\P_x$ has the same law as $(X,\P_{cx})$. The case when $\alpha=2$ corresponds to Brownian motion, which we henceforth exclude from all subsequent discussion. It is known that the Laplace exponent can be parametrised so that
\[
  \psi(i\theta)=-|\theta|^\alpha (e^{i\pi\alpha(1/2-\rho)}\1_{\{\theta\geq 0\}} + e^{-i\pi\alpha(1/2-\hat\rho)}\1_{\{\theta<0\}}), \qquad\theta\in\mathbb{R},
\]
where $\rho:=\P(X_1\geq 0)$ is the positivity parameter and $\hat\rho=1-\rho$. In that case, the two factors are easily identified as $ \kappa(\lambda)=\lambda^{\alpha\rho}$ and $\hat\kappa(\lambda)=\lambda^{\alpha\hat\rho}$ for $\lambda \geq 0$, with associated potentials possessing densities proportional to $x^{\alpha\rho-1}$ and $x^{\alpha\hat\rho - 1}$ respectively for $x\geq 0$. In part, this goes a long way to explaining why so many fluctuation identities are explicitly available for stable processes. 

In this paper, our objective is to produce a new explicit characterisation of a second  type of explicit  Wiener--Hopf factorisation embedded in the $\alpha$-stable process, the so-called `deep factorisation' first considered in \cite{kyprianou2015deep}, through its representation as a real-valued self-similar Markov process. In the spirit of the  interplay between the classical Wiener--Hopf factorisation and fluctuation theory of the underlying L\'evy process, our analysis will produce a collection of new results for stable processes which are based on identities for potentials derived from the deep factorisation. Before going into detail regarding our results concerning the deep factorisation, let us first present some of the new results we shall obtain {\it en route} for stable processes.

\subsection{Results on fluctuations of stable processes}
The first of our results on stable processes concerns the `point of closest reach' to the origin for stable processes with index $\alpha\in(0,1)$. Recall that for this index range, the stable process does not hit points and, moreover, $\liminf_{t\to\infty}|X_s| = \infty$.  Hence, either on the positive or negative side of the origin, the path of the stable process has a minimal radial distance. Moreover, this distance is achieved at the unique time $\underline{m}$  such that $|X_t|\geq |X_{\underline{m}}|$ for all $t\geq 0$.  Note, uniqueness follows thanks to regularity of $X$ for both $(0,\infty)$ and $(-\infty,0)$.

\begin{prop}[Point of closest reach]\label{prop:pointofclosestreach} Suppose that $\alpha \in (0,1)$, then for $x>0$ and $|z|\leq x$,
  \[
    \P_x(X_{\underline m}\in dz) = \frac{\Gamma(1-\alpha\rho) }{\Gamma(1-\alpha)\Gamma(\alpha\hat\rho)}\frac{x+z}{|2z|^\alpha}
    \left(x-|z|\right)^{\alpha\hat\rho-1} \left(x+|z|\right)^{\alpha\rho -1}dz.
  \]
\end{prop}

In the case that $\alpha=1$, the stable process does not hit points and we have that $\limsup_{t\to\infty}|X_t| = \infty$ and $\liminf_{t\to\infty}|X_t| = 0$ and hence it is difficult to produce a result in the spirit of the previous theorem. However,
when $\alpha\in (1,2)$, the stable process will hit all points almost surely, in particular $\tau^{\{0\}}: = \inf\{t>0: X_t = 0\}$ is $\mathbb{P}_x$-almost surely finite for all $x\in\mathbb{R}$. This allows us to talk about the `point of furthest' reach until absorption at the origin. To this end, we define  $\overline{m}$ to be the unique time such that $|X_t|\leq |X_{\overline{m}}|$ for all $t\leq \tau^{\{0\}}$. Note again, that uniqueness is again guaranteed by regularity of the upper and lower half line for  the stable process.

\begin{prop}[Point of furthest reach]\label{prop:pointoffurthestreach}
  Suppose that $\alpha \in (1,2)$, then for each $x>0$ and $|z|>x$,
  \begin{eqnarray*}
    \P_x(X_{\overline m}\in dz) &=&\frac{\alpha-1}{2  |z|^\alpha} \Bigg(|x+z| (|z|-x)^{\alpha\rho-1}(|z|+x)^{\alpha\hat\rho-1}
    \\
  &&\hspace{2cm}\left.-(\alpha-1)x   \int_1^{|z|/x}(t-1)^{\alpha\rho-1}(t+1)^{\alpha\hat\rho-1} dt\right).
  \end{eqnarray*}
\end{prop}

Finally we are also interested in  {\it radially reflected stable processes}:
\[
  R_t = \frac{X_t}{M_t \vee1} \qquad t \geq 0,
\]
where $M_t = \sup_{s \leq t}|X_s|$, $t\geq 0$. It is easy to verify, using the scaling and Markov properties, that for $t,s\geq 0$,
\begin{equation*} M_{t+s}=|X_{t}| \left(\frac{M_{t}}{|X_{t}|}\vee\widetilde{M}_{s|X_{t}|^{-\alpha}}\right),\qquad X_{t+s}=|X_{t}|\widetilde{X}_{s|X_{t}|^{-\alpha}},
\end{equation*}
where $(\widetilde{M}, \widetilde{X})$ is such that, for all bounded measurable functions $f$,
\[
  \mathbb{E}[f(\widetilde{M}_s,\widetilde{X}_s)|\sigma(X_u: u\leq t)] = g({\rm sign}(X_t)),
\]
where
$g(y) = \mathbb{E}_y[f(M_s,X_s)]$.
It follows  that, whilst  the process $R$  is not Markovian, the pair $(R,M)$ is a strong Markov process. In forthcoming work, in the spirit of \cite{MR2917773}, we shall demonstrate how an excursion theory can be developed for the pair $(R,M)$. In particular, one may build a local time process which is supported on the closure of the times $\{t: |R_t| = 1\}$. The times that are not in this supporting set form a countable union of disjoint intervals during which $R$ executes an excursion {\it into} the  interval $(-1,1)$ (i.e. the excursion begins on the boundary $(-1,1)$ and runs until  existing this interval). We go no further into the details of this excursion theory here. However, it is worthy of note that one should expect an ergodic limit of the process $R$ which is bounded in $[-1,1]$. The following result demonstrates this prediction in explicit detail.


\begin{figure}
  \centering
  \begin{subfigure}[t]{0.35\textwidth}
    \centering
    \includegraphics[width=\textwidth]{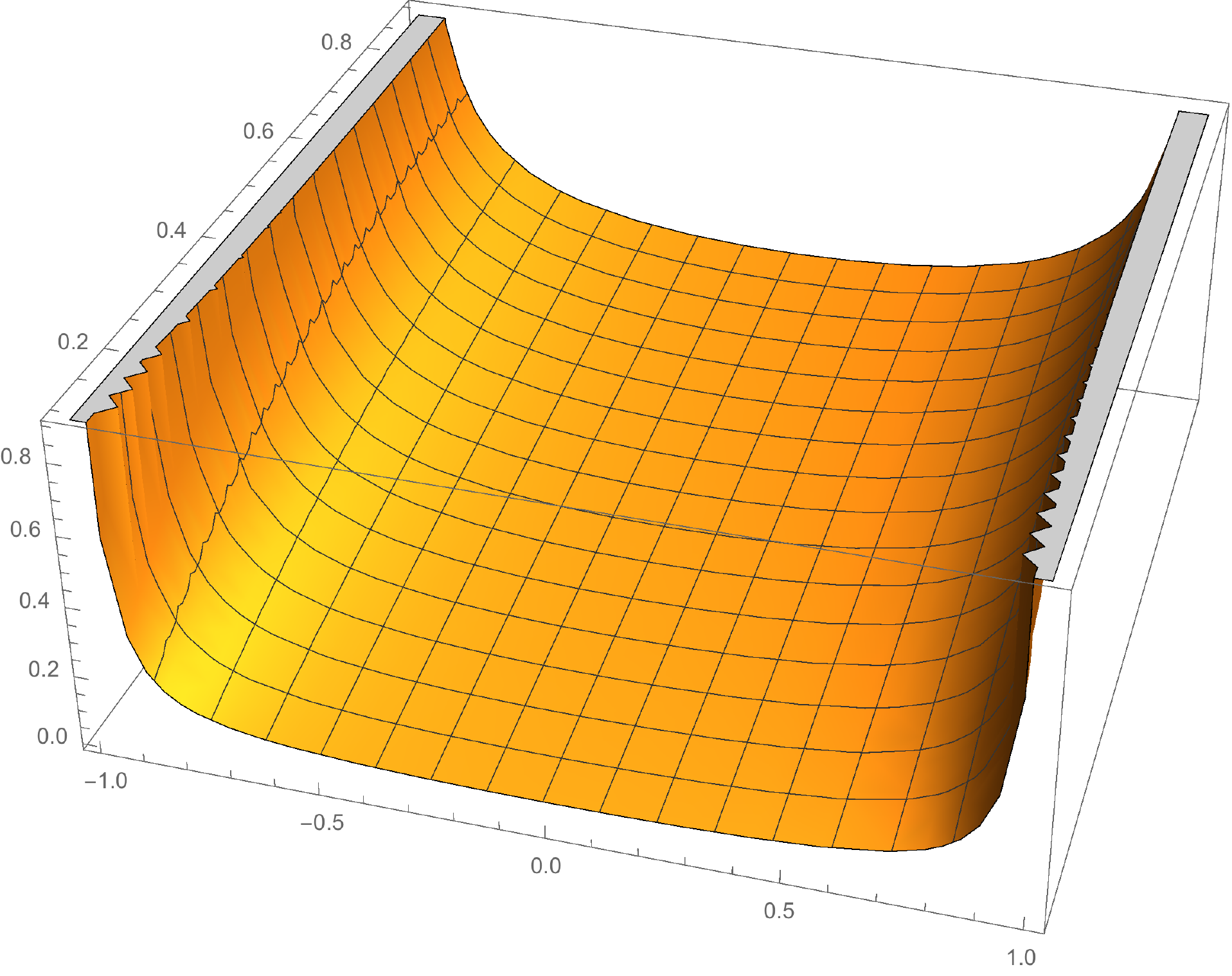}
    \caption{$\rho=1/2$}
  \end{subfigure}
  \hspace{40pt}
  \begin{subfigure}[t]{0.35\textwidth}
    \centering
    \includegraphics[width=\textwidth]{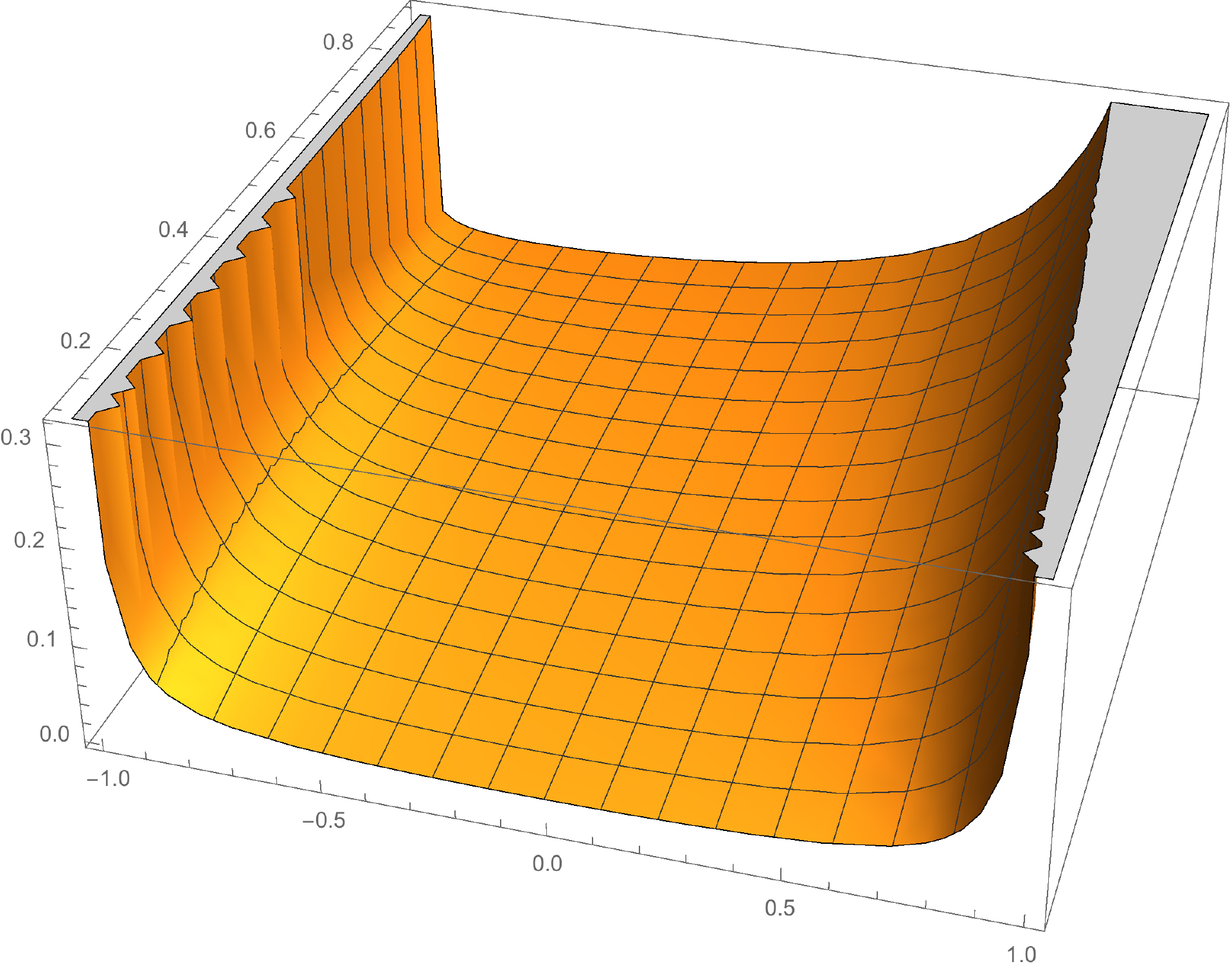}
    \caption{$\rho=9/10$}
  \end{subfigure}

  \caption{The limiting distribution of the reflected process for two different values of $\rho$. Note that the concentration of mass at the values $1,-1$ is a consequence of the time that $X$ spends in close proximity to $M$. }\label{fig:reflected_invariant}
\end{figure}

\begin{theorem}[Stationary distribution of the radially reflected process]\label{thm:reflecting}
  Suppose that $\alpha \in (0,1)$. Let $x \in (-1,1)$, then under $\mathbb{P}_x$,
  $R$ has a limiting  distribution $\mu$, concentrated on $[-1,1],$ given by
 \begin{align*}
   \frac{d\mu(y)}{dy} = 2^{-\alpha}\frac{\Gamma(\alpha)}{\Gamma(\alpha\rho)\Gamma(\alpha\hat\rho)}
   \left[
   (1-y)^{\alpha\hat\rho -1}(1+y)^{\alpha\rho} +(1-y)^{\alpha\hat\rho}(1+y)^{\alpha\rho-1}
   \right],\qquad y\in[-1,1]
\end{align*}
(See Figure \ref{fig:reflected_invariant} which has two  examples of this density for different values of $\rho$.)
  \end{theorem}

  \subsection{Results on the deep factorisation of stable processes}

  In order to present our results on the deep factorisation,  we must first look the Lamperti--Kiu representation of real self-similar Markov processes, and in particular for the case of a stable process.


  A Markov process $(X,\P_x)$ is called a real self-similar Markov processes (rssMp) of index $\alpha>0$ if, for every $c >0$, $(cX_{c^{-\alpha}t}:t \geq 0)$ has the same law as $X$. In particular, an $\alpha$-stable L\'evy process is an example of an rssMp, in which case $\alpha\in(0,2)$. Recall, every rssMp can be written in terms of what is referred to as a Markov additive process (MAP) $(\xi,J)$. The details of this can be found in Section \ref{sec:maps}. Essentially $(\xi,J)$ is a certain (possibly killed) Markov process taking values in $\R\times \{1,-1\}$ and is completely characterised by its so-called MAP--exponent $\bm F$ which, when it exists, is a function mapping $\mathbb C$ to $2\times 2$ complex valued matrices\footnote{Here and throughout the paper the matrix entries are arranged by
    \[
      A=\left(\begin{matrix}
          A_{1,1} & A_{1,-1}\\
          A_{-1,1} & A_{-1,-1}
      \end{matrix}\right).
    \]
  }, which satisfies,
  \[
    (e^{\bm F(z) t})_{i,j}=\bE_{0,i}[e^{z\xi(t)};J_t=j], \qquad i,j = \pm1, t\geq 0.
  \]

  \begin{figure}
    \begin{center}
      \begin{subfigure}[t]{0.27\textwidth}
        \centering
        \includegraphics[width=\textwidth]{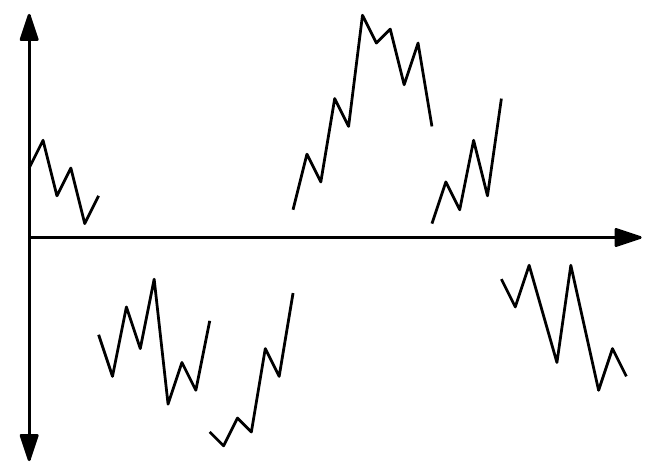}
        \caption{$X$}
      \end{subfigure}
      ~
      \begin{subfigure}[t]{0.27\textwidth}
        \centering
        \includegraphics[width=\textwidth]{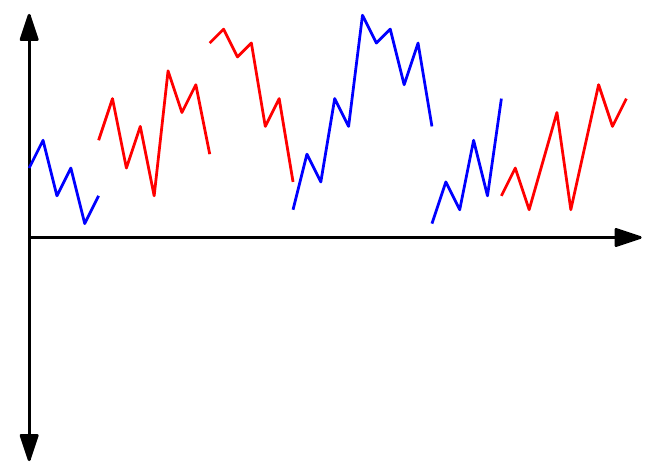}
        \caption{$|X|$}
      \end{subfigure}
      ~
      \begin{subfigure}[t]{0.27\textwidth}
        \centering
        \includegraphics[width=\textwidth]{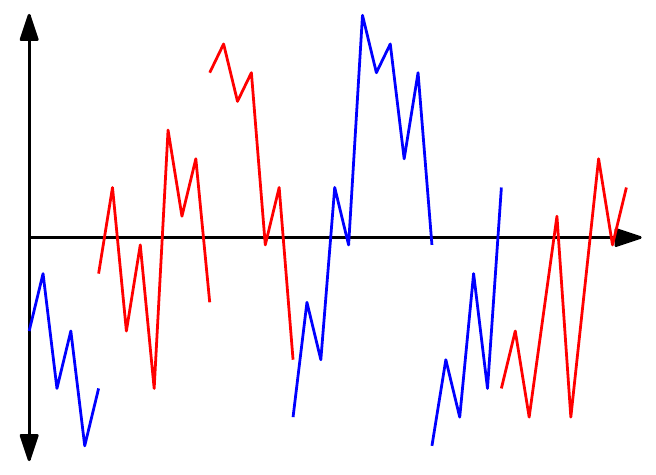}
        \caption{$(\xi,J)$}
      \end{subfigure}
    \end{center}
    \caption{The graphical representation of the Lamperti--Kiu transformation. First figure is the process $X$. In the second figure we take the absolute value of $X$, colouring the positive paths blue and the negative paths red. Then we apply a $\log$ and a time-change to obtain $\xi$, the colours represent the state that $J$ is in.}
    \label{fig:lamperti_kiu}
  \end{figure}

  Next we present the Lamperti--Kiu transfomation which relates a rssMp to a MAP, see also Figure~\ref{fig:lamperti_kiu}.
  It follows directly from~\cite[Theorem 6]{MR3160562}
  \begin{theorem}[Lamperti--Kiu transformation]
    Suppose that $X=(X_t:t\geq 0)$ is a real--valued self--similar Markov process, killed when it hits the origin, then there exists a Markov additive process $(\xi,J)$ on $\R\times \{1,-1\}$ such that started from $X_0=x$,
    \begin{equation}\label{eq:lamp_kiu}
      X_t =\begin{cases}
        |x| e^{\xi_{\varphi(|x|^{-\alpha}t)}} J_{\varphi(|x|^{-\alpha}t)} & \text{if }t <\int_0^\infty e^{\alpha\xi_u}\,{\rm d}u\\
        \partial & \text{if }t \geq \int_0^\infty e^{\alpha\xi_u}\,{\rm d}u
      \end{cases}
    \end{equation}
    where $\partial$ is a cemetery state and
    \[
      \varphi(t):=\inf\left\{s>0:\int_0^s e^{\alpha\xi_u}\,du >t\right\}
    \]
    with convention that $\inf \emptyset = \infty$ and $e^{\xi_\infty}=0$.
    Conversely, every Markov additive process $(\xi,J)$ on $\R\times \{1,-1\}$ defines a real--valued self--similar Markov process, killed when hitting the origin, via \eqref{eq:lamp_kiu}.
  \end{theorem}

  We will denote the law of $(\xi,J)$ started from $(x,i)$ by $\bP_{x,i}$.
  Note that $\xi$ describes the radial part of $X$ and $J$ the sign, and thus a decrease in $\xi$ when $J=1$, for example, corresponds to an increase in $X$. See again Figure~\ref{fig:lamperti_kiu}.

  In the case $X_t\geq 0$ for all $t \geq 0$, we have that $J_t=1$ for all $t \geq 0$. In this special case the transformation is known as the Lamperti transform and the process $\xi$ is a L\'evy process. The Lamperti transformation, introduced in \cite{MR0307358}, has been studied intensively, see for example \cite[Chapter 13]{MR3155252} and references therein. In particular, the law and the Wiener--Hopf factorisation of $\xi$ is known in many cases, for example \cite[Section 13.4]{MR3155252} and \cite{MR2797981}.

  Conversely, very little is known about the general case. In this paper, we shall consider the case when $X$ is an $\alpha$--stable process (until first hitting the origin in the case that $\alpha\in(1,2)$) and note that in that case $(\xi,J)$ is not necessarily a L\'evy process. In this case the MAP--exponent is known to be
  \begin{equation}\label{eq:MAP_expo}
    \bm F({z})=\left(\begin{array}{*2{>{\displaystyle}c}}
        -\frac{\Gamma(\alpha-{z})\Gamma(1+{z})}{\Gamma(\alpha\hat\rho-{z})\Gamma(1-\alpha\hat\rho+{z})} & \frac{\Gamma(\alpha-{z})\Gamma(1+{z})}{\Gamma(\alpha\hat\rho)\Gamma(1-\alpha\hat\rho)}\\
        &\\
        \frac{\Gamma(\alpha-{z})\Gamma(1+{z})}{\Gamma(\alpha\rho)\Gamma(1-\alpha\rho)}& -\frac{\Gamma(\alpha-{z})\Gamma(1+{z})}{\Gamma(\alpha\rho-{z})\Gamma(1-\alpha\rho+{z})}
    \end{array}\right),
  \end{equation}
  for Re$({z}) \in (-1,\alpha)$, and the associated process is called the {\it Lamperti-stable} MAP by analogy to \cite{CC,MR2797981}.
  Notice that the rows of $\bm F(0)$ sum to zero which means the MAP is not killed.

  Similar to the case of L\'evy processes, we can define $\bm\kappa$ and $\hat{\bm\kappa}$ as the Laplace exponent of the ascending and descending ladder height process for $(\xi,J)$, see Section \ref{sec:maps} for more details.
  The analogue of Wiener--Hopf factorisation for MAPs states that, up to pre-multiplying $\bm \kappa$ or $\hat{\bm \kappa}$ (and hence equivalently up to pre-multiplying $\bm F$) by a strictly positive diagonal matrix, we have that
  \begin{equation}\label{eq:factorisation}
    - \bm F(i\lambda) = {\bm\Delta}^{-1}_{\pi}\hat{\bm \kappa}(i\lambda)^T{\bm\Delta}_{\pi}\bm\kappa(-i\lambda),
  \end{equation}
  where
  \begin{equation}\label{eq:delta_matrix}
    {\bm\Delta}_\pi:=\left(\begin{array}{cc}
        \sin(\pi\alpha\rho) & 0 \\
        0 & \sin(\pi\alpha\hat\rho).
    \end{array}\right).
  \end{equation}
  Note, at later stages, during computations, the reader is reminded that, for example, the term $\sin(\pi\alpha\rho)$ is preferentially represented via the reflection identity $\pi/[\Gamma(\alpha\rho)\Gamma(1-\alpha\rho)]$.
  The factorisation in \eqref{eq:factorisation} can be found in \cite{MR650610} and \cite{MR3174223} for example. The exposition in the prequel to this paper, \cite{kyprianou2015deep}, explains in more detail how premultiplication of any of the terms in \eqref{eq:factorisation} by a strictly positive diagonal matrix corresponds to a linear time change in the associated MAP which is modulation dependent. Although this may present some concern to the reader, we note that this is of no consequence to our computations which focus purely on spatial events and therefore the range of the MAPS under question, as opposed to the time-scale on which they are run. Probabilistically speaking, this mirrors a similar situation with the Wiener--Hopf factorisation for L\'evy processes, \eqref{eq:normal_WH}, which can only be determined up to a constant (which corresponds to a linear scaling in time). Taking this into account, our main result identifies the inverse factors $\bm\kappa^{-1}$ and $\hat{\bm\kappa}^{-1}$ explicitly up to post-multiplication by a strictly positive diagonal matrix.

  \begin{theorem}\label{thm:factorisation}
    Suppose that $X$ is an $\alpha$-stable process then we have that, up to post-multiplication by a strictly positive diagonal matrix,  the factors $\bm\kappa^{-1}$ and $\hat{\bm\kappa}^{-1}$ are given as follows. For $a,b,c \in \R$, define
    \begin{equation}\label{eq:Psi}
      \Psi(a,b,c):=\int_0^1 u^{a}(1-u)^{b}(1+u)^{c} du.
    \end{equation}
    \underline{For $\alpha\in(0,1)$:} 
    \begin{align*}
      &\bm\kappa^{-1}(\lambda)=
      \left(\begin{matrix}
          \Psi(\lambda-1,\alpha\rho-1,\alpha\hat\rho) & 
          \Psi(\lambda-1,\alpha\rho,\alpha\hat\rho-1) \\
          &\\
          \Psi(\lambda-1,\alpha\hat\rho,\alpha\rho-1) & 
          \Psi(\lambda-1,\alpha\hat\rho-1,\alpha\rho)
      \end{matrix}\right)
    \end{align*}
    and
    \begin{align*}
      &  \left(\begin{array}{cc}  \frac{\Gamma(\alpha\hat\rho)}{\Gamma(1-\alpha\rho) }&0\\
      0& \frac{\Gamma(\alpha\rho)}{\Gamma(1-\alpha\hat\rho) }
  \end{array}\right)
  \hat{\bm\kappa}^{-1}(\lambda) =
  \left(\begin{matrix}
      \Psi(\lambda-\alpha,\alpha\hat\rho-1,\alpha\rho) & \Psi(\lambda-\alpha,\alpha\hat\rho,\alpha\rho-1) \\
      &\\
      \Psi(\lambda-\alpha,\alpha\rho,\alpha\hat\rho-1) &\Psi(\lambda-\alpha,\alpha\rho-1,\alpha\hat\rho)
  \end{matrix}\right).
\end{align*}
\underline{For $\alpha=1$:}
        \begin{align*}
          \bm\kappa^{-1}(\lambda)=\hat{\bm \kappa}^{-1}(\lambda)
          &=\left(\begin{matrix}
          \Psi(\lambda-1,-1/2,1/2) & \Psi(\lambda-1,1/2,-1/2)\\
          &\\
          \Psi(\lambda-1,1/2,-1/2)& \Psi(\lambda-1,-1/2,1/2)
      \end{matrix}\right).
    \end{align*}
    \underline{For $\alpha \in (1,2)$:}
    \begin{align*}
      \bm\kappa^{-1}(\lambda)
      &=
      \left(\begin{matrix}
          \Psi(\lambda-1,\alpha\rho-1,\alpha\hat\rho) & \Psi(\lambda-1,\alpha\rho,\alpha\hat\rho-1)\\
          &\\
          \Psi(\lambda-1,\alpha\hat\rho,\alpha\rho-1)& \Psi(\lambda-1,\alpha\hat\rho-1,\alpha\rho)
        \end{matrix}\right)\\
        &\hspace{1cm}- \frac{(\alpha-1)}{(\lambda+\alpha-1)}\left(\begin{matrix}
          \Psi(\lambda-1,\alpha\rho-1,\alpha\hat\rho-1) & \Psi(\lambda-1,\alpha\rho-1,\alpha\hat\rho-1)\\
          &\\
          \Psi(\lambda-1,\alpha\hat\rho-1,\alpha\rho-1) & \Psi(\lambda-1,\alpha\hat\rho-1,\alpha\rho-1)
      \end{matrix}\right)
    \end{align*}
    and
    \begin{align*}
      \left(\begin{matrix}
      \frac{\Gamma(\alpha\hat\rho)}{\Gamma(1-\alpha\rho) }& 0 \\
      0 &   \frac{\Gamma(\alpha\rho)}{\Gamma(1-\alpha\hat\rho) } 
    \end{matrix}
  \right)        
  \hat{\bm\kappa}^{-1}(\lambda)&= \left(\begin{matrix}
      \Psi(\lambda-\alpha,\alpha\hat\rho-1,\alpha\rho) & \Psi(\lambda-\alpha,\alpha\hat\rho,\alpha\rho-1)\\
      &\\
      \Psi(\lambda-\alpha,\alpha\rho,\alpha\hat\rho-1)& \Psi(\lambda-\alpha,\alpha\rho-1,\alpha\hat\rho)
  \end{matrix}\right) \\
  &\hspace{0.3cm}- \frac{(\alpha-1)}{(\lambda+\alpha-1)}\left(\begin{matrix}
  \Psi(\lambda-\alpha,\alpha\hat\rho-1,\alpha\rho-1) 
  & 
  \Psi(\lambda-\alpha,\alpha\hat\rho-1,\alpha\rho-1)\\
  &\\
  \Psi(\lambda-\alpha,\alpha\rho-1,\alpha\hat\rho-1) 
  & 
  \Psi(\lambda-\alpha,\alpha\rho-1,\alpha\hat\rho-1)
      \end{matrix}\right).
    \end{align*}

  \end{theorem}

  Note that  the function $\Psi$ can also be written in terms of hypergeometric functions, specifically
  \[
    \Psi(a,b,c) = \frac{\Gamma(a+1)\Gamma(b+1)}{\Gamma(a+b+2)} \,\setlength\arraycolsep{1pt}
    {}_2 F_1\left(-c , a+1 , a+b+2
    ;-1\right),
  \]
  where ${}_2 F_1$ is the usual Hypergeometric function.
  There are many known identities for such hypergeometric functions, see for example \cite{MR1688958}. The appearance of hypergeometric functions is closely tied in with the fact that we are working with stable processes, for example \cite[Theorem 1]{MR3010227} describes the laws of various conditioned stable processes in terms of what are called hypergeometric L\'evy processes.

  The factorisation of $\bm F$ first appeared in Kyprianou \cite{kyprianou2015deep}. Here our factorisation of $\bm F^{-1}$ is completely independent from the derivation in \cite{kyprianou2015deep}, moreover there is no clear way to invert the factors in \cite{kyprianou2015deep} to derive our results. The Bernstein functions that appear in \cite{kyprianou2015deep} have not, to our knowledge, appeared in the literature and are in fact considerably harder to do computations with, whereas the factorisation that appears here is given in terms of well studied hypergeometric functions. Our proof here is much simpler and shorter as it only relies on entrance and exit probabilities of $X$.

  Expressing the factorisation in terms of the inverse matrices has a considerable advantage in that the potential measures of the MAP are easily identified. To do this, we let $\bm u$ denote the unique matrix valued function so that, for $\lambda\geq 0$,
  \[
    \int_0^\infty e^{-\lambda x}\bm u_{i,j}(x) dx = \bm\kappa^{-1}_{i,j}(\lambda) \qquad \text{ for each }i,j=\pm 1.
  \]
  Similarly, let $\hat{\bm u}$ denote the unique matrix valued function so that, for $\lambda\geq 0$,
  \[
    \int_0^\infty e^{-\lambda x}\hat{\bm u}_{i,j}(x) dx = \hat{\bm\kappa}_{i,j}^{-1}(\lambda) \qquad \text{ for each }i,j=\pm 1.
  \]
  The following corollary follows from Theorem \ref{thm:factorisation} by using the substitution $x=-\log u$ in the definition of $\Psi$.

  \begin{cor}\label{cor:potentials}
    The potential densities are given by the following.

    \noindent\underline{For $\alpha\in(0,1)$:}
        \begin{align*}
          &\bm u(x)=
          \left(\begin{matrix}
              (1-e^{-x})^{\alpha\rho-1}(1+e^{-x})^{\alpha\hat\rho} & 
              (1-e^{-x})^{\alpha\rho}(1+e^{-x})^{\alpha\hat\rho-1} \\
              &\\
              (1-e^{-x})^{\alpha\hat\rho}(1+e^{-x})^{\alpha\rho-1} &
              (1-e^{-x})^{\alpha\hat\rho-1}(1+e^{-x})^{\alpha\rho}
          \end{matrix}\right)
        \end{align*}
        and
        \begin{align*}
          &  \left(\begin{array}{cc}  \frac{\Gamma(\alpha\hat\rho)}{\Gamma(1-\alpha\rho) }&0\\
          0& \frac{\Gamma(\alpha\rho)}{\Gamma(1-\alpha\hat\rho) }
      \end{array}\right)\hat{\bm u}(x)=
      \left(\begin{matrix}
          (e^x-1)^{\alpha\hat\rho-1}(e^x+1)^{\alpha\rho}&(e^x-1)^{\alpha\hat\rho}(e^x+1)^{\alpha\rho-1} \\
          &\\
          (e^x-1)^{\alpha\rho}(e^x+1)^{\alpha\hat\rho-1}&(e^x-1)^{\alpha\rho-1}(e^x+1)^{\alpha\hat\rho}
      \end{matrix}\right).
    \end{align*}
    \underline{For $\alpha=1$:}
    \begin{align*}
      \bm u(x)=\hat{\bm u} (x)
      &=\left(\begin{matrix}
      (1-e^{-x})^{-1/2}(1+e^{-x})^{1/2} & (1-e^{-x})^{1/2}(1+e^{-x})^{-1/2}\\
      &\\
      (1-e^{-x})^{1/2}(1+e^{-x})^{-1/2}& (1-e^{-x})^{-1/2}(1+e^{-x})^{1/2}
  \end{matrix}\right).
\end{align*}
\underline{For $\alpha\in(1,2)$:}
    \begin{align*}
      \bm u (x)
      &=
      \left(\begin{matrix}
          (1-e^{-x})^{\alpha\rho-1}(1+e^{-x})^{\alpha\hat\rho} & (1-e^{-x})^{\alpha\rho}(1+e^{-x})^{\alpha\hat\rho-1}\\
          &\\
          (1-e^{-x})^{\alpha\hat\rho}(1+e^{-x})^{\alpha\rho-1}& (1-e^{-x})^{\alpha\hat\rho-1}(1+e^{-x})^{\alpha\rho}
      \end{matrix}\right) \\
      &- (\alpha-1) e^{-(\alpha-1)x}\int_0^{e^x}\left(\begin{matrix}
      (t-1)^{\alpha\rho-1}(t+1)^{\alpha\hat\rho-1} & (t-1)^{\alpha\rho-1}(t+1)^{\alpha\hat\rho-1}\\
      &\\
      (t-1)^{\alpha\hat\rho-1}(t+1)^{\alpha\rho-1} & (t-1)^{\alpha\hat\rho-1}(t+1)^{\alpha\rho-1}
  \end{matrix}\right)\,dt
\end{align*}
and
\begin{align*}
  \left(\begin{matrix}
  \frac{\Gamma(\alpha\hat\rho)}{\Gamma(1-\alpha\rho) }& 0 \\
  0 &   \frac{\Gamma(\alpha\rho)} {\Gamma(1-\alpha\hat\rho) }
          \end{matrix}\right)
          \hat{\bm u}(x)&=\left(\begin{matrix}
              (e^x-1)^{\alpha\hat\rho-1}(e^x+1)^{\alpha\rho}& (e^x-1)^{\alpha\rho}(e^x+1)^{\alpha\rho-1}\\
              &\\
              (e^x-1)^{\alpha\rho}(e^x+1)^{\alpha\hat\rho-1}& (e^x-1)^{\alpha\rho-1}(e^x+1)^{\alpha\hat\rho}
          \end{matrix}\right) \\
          &\hspace{1cm}- (\alpha-1)\int_0^{e^x}\left(\begin{matrix}
          (t-1)^{\alpha\hat\rho-1}(t+1)^{\alpha\rho-1} &(t-1)^{\alpha\hat\rho-1}(t+1)^{\alpha\rho-1}\\
          &\\
          (t-1)^{\alpha\rho-1}(t+1)^{\alpha\hat\rho-1} & (t-1)^{\alpha\rho-1}(t+1)^{\alpha\hat\rho-1}
      \end{matrix}\right)\,dt,
    \end{align*}
    where the integral of a matrix is done component-wise.

  \end{cor}

  Before concluding this section, we also remark that the explicit nature of the factorisation of the Lamperti-stable MAP suggests that other factorisations of MAPs in a larger class of such processes may also exist. Indeed, following the introduction of the Lamperti-stable L\'evy process in \cite{CC}, for which an explicit Wiener--Hopf factorisation are available, it was quickly discovered that many other explicit Wiener--Hopf factorisations could be found by studying related positive self-similar path functionals of stable processes. In part, this stimulated the definition of the class of hypergeometric L\'evy processes for which the Wiener--Hopf factorisation is explicit; see \cite{MR3010227, KPW, KKPvS}. One might therefore also expect  a general class of MAPs to exist, analogous to the class of hypergeometric L\'evy processes, for which a matrix factorisation such as the one presented above, is explicitly available. Should that be the case, then the analogue of fluctuation theory for L\'evy processes awaits further development in concrete form, but now for `hypergeometric' MAPs. See for example some of the general fluctuation theory for MAPs that appears in the Appendix of \cite{dereich2015real}.

  \subsection{Outline of the paper}

  The rest of the paper is structured as follows.  In Section \ref{sec:maps} we introduce some technical background material for the paper. Specifically, we introduce Markov additive processes (MAPs) and ladder height processes for MAPs in more detail.
  We then prove the results of the paper by separating into three cases.
  In Section \ref{sec:proof_a_less} we show Theorem \ref{thm:factorisation} for $\alpha \in (0,1)$, and Proposition \ref{prop:pointofclosestreach}. In Section \ref{sec:proof_a_great} we prove Theorem \ref{thm:factorisation} for $\alpha \in (1,2)$, and Proposition \ref{prop:pointoffurthestreach}. In Section \ref{sec:proof_a_1} we show Theorem \ref{thm:factorisation} for $\alpha=1$. Finally in Section \ref{sec:reflecting} we prove Theorem \ref{thm:reflecting}.


  \section{Markov additive processes}\label{sec:maps}

  In this section we will work with a (possibly killed) Markov processes $(\xi,J)=((\xi_t,J_t):t \geq 0)$ on $\R \times \{1,-1\}$. For convenience, we will always assume that $J$ is irreducible on $\{1,-1\}$. For such a process $(\xi,J)$  
  we let $\bP_{x,i}$ be the law of $(\xi,J)$ started from the state $(x,i)$.

  \begin{defn}
    A Markov process $(\xi,J)$ is called a Markov additive process (MAP) on $\mathbb{R}\times\{1,-1\}$ if, for any $t \geq 0$ and $j=-1,1$, given $\{J_t=j\}$, the process $((\xi_{s+t}-\xi_t,J_{s+t}):s\geq 0)$ has the same law as $(\xi,J)$ under $\bP_{0,j}$.
  \end{defn}

  The topic of MAPs are covered in various parts of the literature. We reference \cite{Cinlar1, Cinlar2, MR889893,MR2766220,MR3160562,MR3174223} to name but a few of the many texts and papers. It turns out that a MAP on $\mathbb{R}\times\{1,-1\}$ requires  five characteristic  components: two independent and  L\'evy processes (possibly killed but not necessarily with the same rates), say $\chi_{1}=(\chi_{1}(t):t \geq 0)$ and $\chi_{-1}=(\chi_{-1}(t):t \geq 0)$, two independent random variables, say  $\Delta_{-1,1}$ and $\Delta_{1,-1}$ on $\R$ and a $2\times 2$ intensity matrix, say  $\bm Q=(q_{i,j})_{i,j=\pm 1}$. We call the quintuple $(\chi_{1},\chi_{-1},\Delta_{-1,1},\Delta_{1,-1},\bm Q)$ the driving factors of the MAP.

  \begin{defn}
    A Markov additive processes on $\mathbb{R}\times\{1,-1\}$ with driving factors $(\chi_{1},\chi_{-1},\Delta_{-1,1},\Delta_{1,-1},\bm Q)$ is defined as follows. Let $J=(J(t):t\geq 0)$ be a continuous time Markov process on $\{1,-1\}$ with intensity matrix $\bm Q$. Let $\sigma_1,\sigma_2,\dots$ denote the jump times of $J$. Set $\sigma_0=0$ and $\xi_0=x$, then for $n \geq 0$ iteratively define
    \[
      \xi(t)=\1_{n>0}(\xi(\sigma_n-) + U^{(n)}_{J(\sigma_n-) , J(\sigma_n)} ) + \chi^{(n)}_{J(\sigma_n)}(t-\sigma_n) \qquad t \in [\sigma_n, \sigma_{n+1}),
    \]
    where $(U^{(n)}_{i,j})_{n \geq 0}$ and $\chi^{(n)}_{i}$ are i.i.d. with distributions $\Delta_{i,j}$ and $\chi_{i}$ respectively.
  \end{defn}

  It is not hard to see that the construction above results in a MAP. Conversely we have that every MAP arises in this manner, we refer to \cite[XI.2a]{MR889893} for a proof.

  \begin{prop}
    A Markov process $(\xi,J)$ is a Markov additive process on $\R\times\{1,-1\}$ if and only if there exists a  quintuple of driving factors $(\chi_{1},\chi_{-1},\Delta_{-1,1},\Delta_{1,-1},\bm Q)$. Consequently, every Markov additive process on $\mathbb{R}\times\{1,-1\}$ can be identified uniquely by a quintuple $(\chi_{1},\chi_{-1},\Delta_{-1,1},\Delta_{1,-1},\bm Q)$ and every quintuple defines a unique Markov additive process.
  \end{prop}

  Let $\psi_{-1}$ and $\psi_1$ be the Laplace exponent of $\chi_{-1}$ and $\chi_1$ respectively (when they exist). For ${z} \in \mathbb C$, let $\bm G({z})$ denote the matrix whose entries are given by $\bm G_{i,j}({z})=\bE[e^{{z} \Delta_{i,j}}]$ (when they exists), for $i \neq j$ and $\bm G_{i,i}({z})=1$. For ${z} \in \mathbb C$, when it exists, define
  \begin{equation}\label{eq:F_def}
    \bm F({z}):= \text{diag}(\psi_{1}({z}),\psi_{-1}({z})) - \bm Q \circ \bm G({z}),
  \end{equation}
  where diag$(\psi_{1}({z}),\psi_{-1}({z}))$ is the diagonal matrix with entries $\psi_{1}({z})$ and $\psi_{-1}({z})$, and $\circ$ denotes element--wise multiplication.
  It is not hard to check that $\bm F$ is unique for each quintuple $(\chi_{1},\chi_{-1},\Delta_{-1,1},\Delta_{1,-1},\bm Q)$ and furthermore, see for example \cite[XI, Proposition 2.2]{MR1978607}, for each $i,j = \pm 1$ and $t \geq 0$,
  \[
    \bE_{0,i}[e^{{z} \xi_t}; J_t=j] = (e^{t\bm F({z})})_{i,j}
  \]
  where $e^{t\bm F({z})}$ is the exponential matrix of $t\bm F({z})$. For this reason we refer to $\bm F$ as a MAP-exponent.

  \subsection{Ladder height process}\label{reflected}

  Here we will introduce the notion of the ladder height processes for MAPs and introduce the matrix Wiener--Hopf factorisation. It may be useful for the reader to compare this to the treatment of these topics for L\'evy processes in \cite[Chapter 6]{MR3155252}.

  Let $(\xi,J)$ be a MAP and define the process $\bar \xi=(\bar \xi_t:t \geq 0)$ by setting $\bar \xi_t=\sup_{s \leq t}\xi_s$. Then it can be shown (see \cite[Theorem 3.10]{MR650610} or \cite[Chapter IV]{MR1406564}) that  there exists two non-constant increasing processes $\bar L^{(-1)}=(\bar L^{(-1)}_t:t \geq 0)$ and $\bar L^{(1)}=(\bar L^{(1)}_t:t \geq 0)$ such that $\bar L^{(-1)}$ increases on the closure of the set $\{t: (\xi_t,J_t)=(\bar\xi_t,-1)\}$ and $\bar L^{(1)}$ increases on the closure of the set $\{t: (\xi_t,J_t)=(\bar\xi_t,1)\}$. Moreover $\bar L^{(-1)}$ and $\bar L^{(1)}$ are unique up to a constant multiples. We call $\bar L=\bar L^{(-1)}+\bar L^{(1)}$ the local time at the maximum.
It may be the case that $\bar L_\infty<\infty$, for example if $\xi$ drifts to $- \infty$. In such a case, both $\bar L^{(-1)}_\infty$ and $\bar L^{(1)}_\infty$ are distributed exponentially. Since the processes $\bar L^{(-1)}$ and $\bar L^{(1)}$ are unique up to constants, we henceforth assume that whenever $\bar L_\infty<\infty$,  the normalisation has been chosen so that
  \begin{equation}\label{eq:killing_ass}
    \text{ both } \bar L^{(-1)}_\infty \text{ and } \bar L^{(1)}_\infty \text{ are distributed as exponentials with rate } 1 .
  \end{equation}
  The ascending ladder height processes $H^+=(H^+_t:t \geq 0)$ is defined as
  \[
    (H^+_t, J^+_t):= (\bar \xi_{\bar L^{-1}_t}, J_{\bar L^{-1}_t}) \qquad t \in [0,\bar L_\infty),
  \]
  where the inverse in the above equation is taken to be right continuous. At time $\bar L_\infty$ we send $(H^+,  J^+)$ to a cemetery state and declare the process killed. It is not hard to see that $(H^+,J^+)$ is itself a MAP.
\[
\text{ We denote by $ \bm \kappa$  the Laplace exponent of } (H^+,J^+), 
\]
that is to say
\[
({\rm e}^{-\bm\kappa(\lambda)t})_{i,j} :=\mathbf{E}_{0,i}[{\rm e}^{-\lambda H^+_t}; J^+_t =j], \qquad \lambda \geq0.
\]
  Similarly, we define $(H^-,J^-)$, called the descending ladder height, by using $- \xi$ in place of $\xi$. We denote by $\underline{\bm\kappa}$ the MAP--exponent of $(H^-,J^-)$.
  Recalling that $\bm\kappa$ can only be identified up to pre-multiplication by a strictly positive diagonal matrix, the choice of normalisation in the local times \eqref{eq:killing_ass} is equivalent to choosing a normalisation $\bm\kappa$.

  For a L\'evy process, its dual is simply given by its negative. The dual of a MAP is a little bit more involved. Firstly, since $J$ is assumed to be irreducible on $\{1,-1\}$, it follows that it is reversible with respect to a unique stationary distribution $\pi=(\pi_{1},\pi_{-1})$. We denote by $\hat{\bm Q}=(\hat q_{i,j})_{i,j=\pm 1}$ the matrix whose entries are given by
  \[
    \hat q_{i,j} = \frac{\pi_j}{\pi_i}q_{j,i}.
  \]
  The MAP--exponent, $\hat{\bm F}$, of the dual $(\hat \xi, \hat J)$ is given by
  \begin{equation}\label{eq:F_hat}
    \hat{\bm F}({z})= \text{diag}(\psi_{1}(-{z}),\psi_{-1}(-{z})) - \hat{\bm Q} \circ \bm G(-{z}),
  \end{equation}
  whenever the right-hand side exists. 
  The duality in this case corresponds to time-reversing $(\xi,J)$, indeed, as observed in \cite[Lemma 21]{dereich2015real}, for any $t\geq 0$,
  \[
    (((\xi_{(t-s)-}-\xi_t,J_{(t-s)-}):0\leq s\leq t),\bP_{0,\pi}) \overset{d}{=} (((\hat \xi_s,\hat J_s): 0\leq s\leq t),\hat\bP_{0,\pi})
  \]
  where we define $(\xi_{0-},J_{0-})=(\xi_0,J_0)$.

  Next define
  \[
    {\bm\Delta}_\pi= \left(\begin{array}{*2{>{\displaystyle}c}}
        \pi_{1} & 0\\
        0 & \pi_{-1}
    \end{array}\right).
  \]
  Then the following lemma follows immediately from \eqref{eq:F_hat}.

  \begin{lemma}\label{lemma:F_hat_transform}
    For each ${z} \in \mathbb C$,
    \[
      \hat{\bm F}({z}) = {\bm\Delta}_\pi^{-1}\bm F(-{z})^T{\bm\Delta}_\pi,
    \]
    where $\bm F(-{z})^T$ denotes the transpose of $\bm F(-{z})$.
  \end{lemma}

  \begin{rem}
    Notice that
    \[
      \bm F(0)=\left(\begin{matrix}
          q_{1,1} & -q_{1,1} \\
          -q_{-1,-1} & q_{-1,-1}
      \end{matrix}\right) = -\bm Q.
    \]
    Hence the matrix ${\bm\Delta}_\pi$ can be computed leading to the form in  \eqref{eq:delta_matrix}. Also note that it is sufficient to use a constant multiple of the matrix ${\bm\Delta}_\pi$.
  \end{rem}

  Similarly to how we obtained $(H^+,J^+)$, we denote by $(\hat H^+,\hat J^+)$ the ascending ladder height process of the dual MAP  $(\hat \xi, \hat J)$.
  \[
    \text{We denote by } \hat{\bm \kappa} \text{ the Laplace exponent of } (\hat H^+,\hat J^+).
  \]

  \begin{lemma}\label{lemma:hat_from_des}
Let $ \underline{\bm\kappa}$ be the matrix exponent of the ascending ladder height processes   of the MAPs $(-\xi,J)$. Then we have, up to post-multiplication by a strictly diagonal matrix, 
    \begin{equation*}
      \underline{\bm\kappa}(\lambda)=\hat{\bm \kappa}(\lambda)\qquad \lambda \geq 0.
    \end{equation*}
  \end{lemma}
  \begin{proof}
    The MAP--exponent $\hat{\bm F}$ of $(\hat\xi,\hat J)$ is given explicitly in~\cite[Section 7]{kyprianou2015deep} and it is not hard to check that  $\bm F(-z)=\hat{\bm F}(z)$. As a consequence, the MAP $(-\xi,J)$ is equal in law to $(\hat{\xi}, \hat{J})$.
    Since $\underline{\bm\kappa}$ and $\hat{\bm\kappa}$ are the matrix Laplace exponent of the ascending ladder height processes   of the MAPs $(-\xi,J)$ and $(\hat{\xi}, \hat{J})$, respectively, it follows that $\underline{\bm\kappa}(\lambda) = \hat{\bm\kappa}(\lambda)$ as required.
  \end{proof}

%
%

  %
  %
  %
  We complete this section by remarking  that if $X$ is an rssMp with Lamperti--Kiu exponent $(\xi,J)$, then $\xi$ encodes the radial distance of $X$ and $J$ encodes the sign of $X$. Consequently if $(H^+,J^+)$ is the ascending ladder height process of $(\xi, J)$, then $H^+$ encodes the supremum of $|X|$ and $J^+$ encodes the sign of where the supremum is reached. Similarly if $(H^-,J^-)$ is the descending ladder height process of $(\xi, J)$, then $H^-$ encodes the infimum of $|X|$ and $J^-$ encodes the sign of where the infimum is reached.

  \begin{figure}
    \begin{center}
      \begin{subfigure}[t]{0.27\textwidth}
        \centering
        \includegraphics[width=\textwidth]{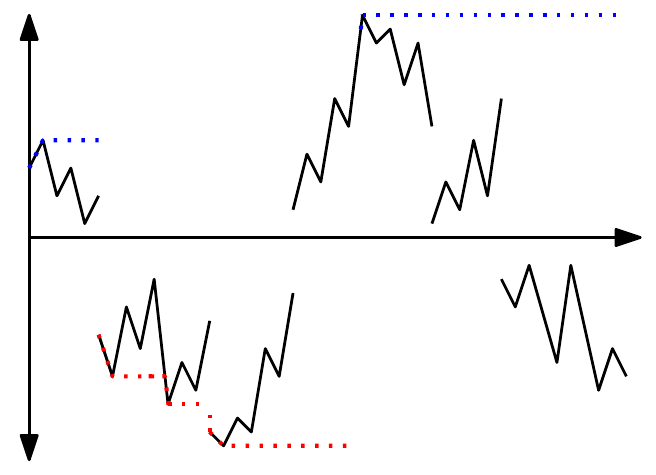}
        \caption{$X$}
      \end{subfigure}
      ~
      \begin{subfigure}[t]{0.27\textwidth}
        \centering
        \includegraphics[width=\textwidth]{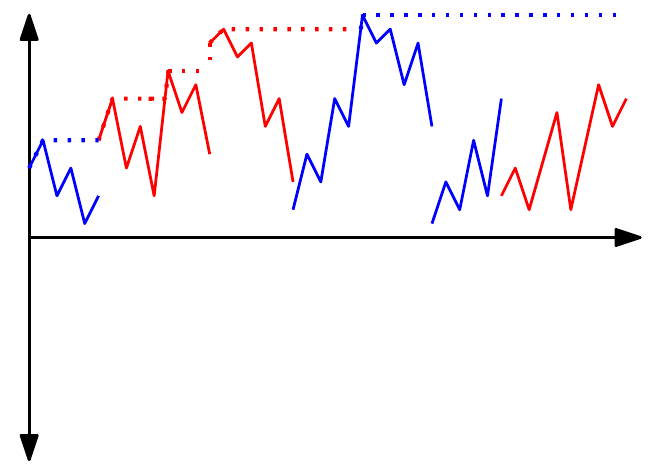}
        \caption{$|X|$}
      \end{subfigure}
      ~
      \begin{subfigure}[t]{0.27\textwidth}
        \centering
        \includegraphics[width=\textwidth]{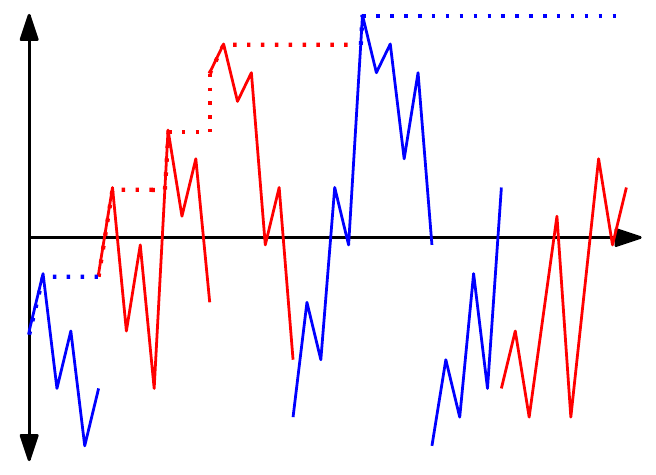}
        \caption{$(\xi,J)$}
      \end{subfigure}
    \end{center}
    \caption{A visualisation of the ladder height process $(H^+,J^+)$. The colours represent the state of $J$. This figure is the reverse of the process described in Figure~\ref{fig:reflected_invariant}.}
    \label{fig:ladder_height}
  \end{figure}

  Although it is not so obvious, one can obtain $\hat{\bm\kappa}$ from $\bm\kappa$ as given by the following lemma which is a consequence of particular properties of the stable process. The proof can be found in  \cite[Section 7]{kyprianou2015deep}.

  \begin{lemma}\label{lemma:K_hat_transform}
    For each $\lambda \geq 0$,
    \[
      \hat{\bm \kappa}(\lambda)={\bm\Delta}^{-1}_\pi \bm\kappa(\lambda +1 -\alpha)|_{\rho \leftrightarrow \hat\rho} {\bm\Delta}_\pi,
    \]
    where $|_{\rho\leftrightarrow\hat\rho}$ indicates exchanging the roles of  $\rho$ with $\hat\rho$.
  \end{lemma}

  \section{Results for \texorpdfstring{$\alpha \in(0,1)$}{a in (0,1)}}\label{sec:proof_a_less}

  In this section we will prove Theorem \ref{thm:factorisation} for $\alpha \in (0,1)$ and Proposition \ref{prop:pointofclosestreach}.

  Suppose that $(X,\P_x)$ is a $\alpha$-stable process started at $x \neq 0$ with $\alpha < 1$ and let $(\xi,J)$ be the MAP in the Lamperti--Kiu transformation of $X$. Let $(H^-,J^-)$ be the descending ladder height process of $(\xi,J)$ and define $\bm U^-$ by
  \[
    \bm U^-_{i,j}(dx) = \int_0^\infty \bP_{0,i}(H^-_t \in dx; J^-_t = j; t< \underline L_\infty)d t, \qquad x\geq 0.
  \]
  Note that we set $\bP_{0,i}(H^-_t \in dx; J^-_t = j)=0$ if $(H^-,J^-)$ is killed prior to time $t$. The measure $\bm U^-_{i,j}$ is related to the exponent $ \underline{\bm\kappa}$ by the relation
    \[
      \int_0^\infty e^{-\lambda x} \bm U_{i,j}^-(dx) = \int_0^\infty\bE_{0,i}[e^{-\lambda H^-_t}; J^-_t = j]dt = \underline{\bm\kappa}^{-1}_{i,j}(\lambda), \qquad \lambda\geq0.
    \]  
  
  
  \noindent We present an auxiliary result.

  \begin{lemma}\label{lemma:u_minus}
    For an $\alpha$-stable process with $\alpha\in(0,1)$ we have that the measure $\bm U^-_{i,j}$ has a density, say $\bm u^-$, such that
    \begin{align*}
      &\bm u^-(x)=   \left(\begin{array}{cc}  \frac{\Gamma(1-\alpha\rho) }{\Gamma(\alpha\hat\rho)}&0\\
      0& \frac{\Gamma(1-\alpha\hat\rho) }{\Gamma(\alpha\rho)} 
      \end{array}\right)
      \left(\begin{array}{*2{>{\displaystyle}c}}
      (e^x-1)^{\alpha\hat\rho-1}(e^x+1)^{\alpha\rho} & (e^x-1)^{\alpha\hat\rho}(e^x+1)^{\alpha\rho-1} \\
      &\\
     (e^x-1)^{\alpha\rho}(e^x+1)^{\alpha\hat\rho-1} & (e^x-1)^{\alpha\rho-1}(e^x+1)^{\alpha\hat\rho}
  \end{array}\right), \qquad x\geq 0.
\end{align*}
  \end{lemma}

  Now we show that Theorem \ref{thm:factorisation} follows from Lemma \ref{lemma:u_minus}.

  \begin{proof}[Proof of Theorem \ref{thm:factorisation} for $\alpha \in(0, 1)$]
    Thus from Lemma \ref{lemma:u_minus}, we can take Laplace transforms to obtain e.g., for $i,j=-1$ and $\lambda\geq0$,
    \begin{align*}
      \underline{\bm\kappa}^{-1}_{-1,-1}(\lambda) &=\frac{\Gamma(1-\alpha\hat\rho) }{\Gamma(\alpha\rho)} \int_0^\infty e^{-\lambda x}(e^x-1)^{\alpha\rho-1}(e^x+1)^{\alpha\hat\rho}dx \\
      &= \frac{\Gamma(1-\alpha\hat\rho) }{\Gamma(\alpha\rho)} \int_0^1 u^{\lambda - \alpha}(1-u)^{\alpha\rho-1}(1+u)^{\alpha\hat\rho}du\\
      &=\frac{\Gamma(1-\alpha\hat\rho) }{\Gamma(\alpha\rho)} \Psi(\lambda-\alpha, \alpha\rho-1,\alpha\hat\rho),
    \end{align*}
    where we have used the substitution $u=e^{-x}$. {\color{black}Once the remaining components of  $\underline{\bm\kappa}^{-1}$ have been obtained similarly to above, we use Lemma~\ref{lemma:hat_from_des} to get $\hat{\bm\kappa}^{-1}$ and then apply Lemma \ref{lemma:K_hat_transform} to get $\bm\kappa^{-1}$. The reader will note that a direct application of the aforesaid Lemma will not give the representation of $\bm\kappa^{-1}$ stated in Theorem \ref{thm:factorisation} but rather the given representation post-multiplied by the diagonal matrix 
    \[
     \left(\begin{array}{cc}  \frac{\Gamma(1-\alpha\rho) }{\Gamma(\alpha\hat\rho)}&0 \\
     0& \frac{\Gamma(1-\alpha\hat\rho) }{\Gamma(\alpha\rho)} 
      \end{array}\right),
    \]
    and this is a because of the normalisation of local time chosen in \eqref{eq:killing_ass}. Note that this is not important for the statement of Theorem \ref{thm:factorisation} as no specific normalisation is claimed there.
     The details of the computation are left out.}
  \end{proof}

  We are left to prove Lemma \ref{lemma:u_minus}. We will do so by first considering the process $X$ started at $x>0$. The case when $x<0$ will follow by  considering the dual $\hat X=-X$.

  Recall that
  $\underline m$ is the unique time such that
  \[
    |X_t| \geq |X_{\underline m}| \qquad \text{ for all } t \geq 0.
  \]
  Our proof relies on the analysis of the random variable $X_{\underline m}$. Notice that when $X_0=x$, $X_{\underline m}$ may be positive or negative and takes values in $[-x,x]$.

  Before we derive the law of $X_{\underline m}$, we first quote the following lemma which appears in \cite[Corollary 1.2]{MR3161489}.
  \begin{lemma}\label{lemma:avoid_strip}
    Let $\tau^{(-1,1)}:=\inf\{t \geq 0: |X_t|<1\}$. We have that, for $x>1$,
    \[
      \P_x(\tau^{(-1,1)}=\infty) = \Phi(x),
    \]
    where
    \[
      \Phi(x) =\frac{\Gamma(1-\alpha\rho)}{\Gamma(\alpha\hat\rho)\Gamma(1-\alpha)} \int_0^{(x-1)/(x+1)} t^{\alpha\hat\rho-1}(1-t)^{-\alpha} \, dt.
    \]
  \end{lemma}

  Lemma \ref{lemma:avoid_strip} immediately gives that the law of $|X_{\underline m}|$ as
  \[
    \P_x(|X_{\underline m}| > z) = \Phi(x/z), \qquad \text{ for } z\in [0,x].
  \]
  Indeed, the event $\{|X_{\underline m}| > z\}$ occurs if and only if $\tau^{(-z,z)}=\infty$. From the scaling property of $X$ we get that $\P_x(\tau^{(-z,z)}=\infty)=\P_{x/z}(\tau^{(-1,1)}=\infty)=\Phi(x/z)$.

  We first begin to derive the law of $X_{\underline m}$ which shows Proposition \ref{prop:pointofclosestreach}.

  \begin{proof}[Proof of Proposition \ref{prop:pointofclosestreach}]
    Fix $x>0$. Similarly to the definition of $\underline m$, we define $\underline m^+$ and $\underline m^-$ as follows:  Let $\underline m^+$ be the unique time such that $X_{\underline m^+}>0$ and
    \[
      X_t \geq X_{\underline m^+} \qquad \text{ for all } t \geq 0 \text{ such that } X_t>0.
    \]
    Similarly let $\underline m^-$ be the unique time such that $X_{\underline m^-}<0$ and
    \[
      X_t \leq X_{\underline m^-} \qquad \text{ for all } t \geq 0 \text{ such that } X_t < 0.
    \]
    In words, $\underline m^+$ and $\underline m^-$ are the times when $X$ is at the closest point to the origin on the positive and negative side of the origin, respectively. Consequently, we have that $X_{\underline m}>0$ if and only if $X_{\underline m^+} < | X_{\underline m^-}|$.
    We now have that
    \[
      \P(| X_{\underline m^-}|> u; X_{\underline m^+} > v )=\P_x(\tau^{(-u,v)}=\infty) =\Phi\left(\frac{2x+u-v}{u+v}\right),
    \]
    where $\Phi$ is defined in Lemma \ref{lemma:avoid_strip} and in the final equality we have scaled space and used the self-similarity of $X$.

    Next we have that for $z \geq 0$,
    \begin{align}\label{eq:minreach_phi}
      \frac{\P_x(X_{\underline m} \in \d z)}{\d z} = -\frac{\partial}{\partial v}\P_x(| X_{\underline m^-}|> z; X_{\underline m^+} > v )|_{v=z}&= -\frac{\partial}{\partial v} \Phi\left(\frac{2x+z-v}{z+v}\right)|_{v=z}\nonumber\\
      &= \frac{x+z}{2z^2}\Phi'\left(\frac{x}{z}\right).
    \end{align}
    The proposition for $z>0$ now follows from an easy computation. The result for $z<0$ follows similarly.
  \end{proof}


  Now we will use \eqref{eq:minreach_phi} to show Lemma \ref{lemma:u_minus}. We will need the following simple lemma which appears in the Appendix of \cite{dereich2015real}.

  \begin{lemma}\label{lemma:crossing_U}
    Let $T_0^-:=\inf\{t \geq 0: \xi_t <0\}$. Under the normalisation \eqref{eq:killing_ass},  for $i,j=-1,1$ and $y>0$, 
    \[
      \bP_{y,i}(T_0^-=\infty; J_{\varphi(\underline m)} = j)=\bm U^-_{i,j}(y):= \bm U^-_{i,j}[0,y].
    \]
  \end{lemma}

  The basic intuition behind this lemma can be described in terms of the descending ladder MAP subordinator $(H^-, J^-)$. The event $\{T_0^-=\infty; J_{\varphi(\underline m)} = j\}$ under $ \bP_{y,i}$ corresponds  the terminal height of $H^-$ immediately prior to being killed being of type $j$ and not reaching the height $y$.  This is expressed precisely by the quantity $\bm U^-_{i,j}(y)$. It is also important to note here and at other places in the text that $\xi$ is regular for both $(0,\infty)$ and $(-\infty,0)$. Rather subtly, this allows us to conclude that the value of $J_{\varphi(\underline{m})} = \lim_{s\uparrow \varphi(\underline{m})}J_s$, or, said another way, the process $\xi$ does not jump away from its infimum as a result of  a change in modulation (see~\cite{ivanovs2015splitting} for a discussion about this). 

\begin{proof}[Proof of Lemma \ref{lemma:u_minus}]
  Let us now describe the event $\{T_0^-=\infty; J_{\varphi(\underline m)} = j\}$ in terms of the underlying process $X$. The event $\{T_0^-=\infty; J_{\varphi(\underline m)} = 1\}$ occurs if and only if $\tau^{(-1,1)}=\infty$ and furthermore the point at which $X$ is closest to the origin is positive, i.e. $X_{\underline m} > 0$. Thus $\{T_0^-=\infty; J_{\varphi(\underline m)} = 1\}$ occurs if and only if $X_{\underline m} > 1$. Using Lemma \ref{lemma:crossing_U} and \eqref{eq:minreach_phi} we have that
  \begin{align*}
    \bm U^-_{1,1}(x)&= \bP_{x,1}(T_0^-=\infty; J_{\varphi(\underline m)} = 1)\\
    &=\P_{e^x}(X_{\underline m} > 1)\\
    &= \frac{1}{2}\int_{1}^{e^x} (e^x+z) \Phi'\left(\frac{e^x}{z}\right) z^{-2}dz\,\\
    &= \frac{1}{2} \int_1^{e^x} (1+1/u)\Phi'(u)du,
  \end{align*}
  where in the final equality we have used the substitution $u=e^{x}/z$.
  Differentiating the above equation we get that 
  \begin{align}\label{eq:u_phi_prime}
    \bm u^-_{1,1}(x)&= \frac{1}{2}(e^{x} + 1)\Phi'(e^x)=\frac{\Gamma(1-\alpha\rho) }{2^\alpha\Gamma(1-\alpha) \Gamma(\alpha\hat\rho)} (e^x-1)^{\alpha\hat\rho-1}(e^x+1)^{\alpha\rho}.
  \end{align}

  Similarly considering the event $\{T_0^+=\infty; J_{\varphi(\underline m)} = -1\}$ we get that
  \[
    \bm u^-_{1,-1}(x) = \frac{1}{2}(e^{x} - 1)\Phi'(e^x) = \frac{\Gamma(1-\alpha\rho) }{2^\alpha\Gamma(1-\alpha) \Gamma(\alpha\hat\rho)}(e^x-1)^{\alpha\hat\rho}(e^x+1)^{\alpha\rho-1}.
  \]

  Notice now that $\bm u^-_{1,j}$ only depends on $\alpha$ and $\rho$. Consider now the dual process $\hat X =(-X_t:t \geq 0)$. This process is the same as $X$ albeit $\rho\leftrightarrow \hat\rho$. To derive the row $\bm u^-_{-1,j}$ we can use $\hat X$ in the computations above and this implies that $\bm u^-_{-1,j}$ is the same as $\bm u^-_{1,-j}$ but exchanging the roles of $\rho$ with $\hat\rho$. This concludes the proof of Lemma \ref{lemma:u_minus}.
  \end{proof}

  \section{Proof of Theorem \ref{thm:factorisation} for \texorpdfstring{$\alpha \in (1,2)$}{a in (1,2)}}\label{sec:proof_a_great}
  In this section we will prove Theorem \ref{thm:factorisation} for $\alpha \in (1,2)$. Let $X$ be an $\alpha$-stable process with $\alpha \in (1,2)$ and let $(\xi,J)$ be the MAP associated to $X$ via the Lamperti--Kiu transformation. The notation and proof given here are very similar to that of the case when $\alpha<1$, thus we skip some of the details.

  Since $\alpha \in (1,2)$ we have that $\tau^{\{0\}}:=\inf\{t \geq 0: X_t=0\}< \infty$ and $X_{\tau^{\{0\}}-} = 0$ almost surely. Hence it is the case that $\xi$ drifts to $-\infty$. Recall that $\overline m$ as the unique time for which $\overline m < \tau^{\{0\}}$ and
  \[
    |X_{\overline m}| \geq |X_t| \qquad \text{ for all } t < \tau^{\{0\}}
  \]
  where the existence of such a time follows from the fact that $X$ is a stable process and so $0$ is regular for $(0,\infty)$ and $(-\infty,0)$.

  The quantity we are interested in is $X_{\overline m}$. We begin with the following lemma, which is lifted from \cite[Corollary 1]{profeta2015harmonic} and also can be derived from the potential given in \cite[Theorem 1]{kyprianou2014potentials}.

  \begin{lemma}\label{lemma:upper_crossing}
    For every $x \in (0,1)$ and $y \in (x,1)$,
    \[
      \P_x(\tau^{\{y\}}<\tau^{(-1,1)^c}) = (\alpha-1)\left(\frac{x-y}{1-y^2}\right)^{\alpha-1}\bar\Phi\left(\left|\frac{1-xy}{x-y}\right|\right),
    \]
    where
    \[
      \bar\Phi(z) = \int_1^{z}(t-1)^{\alpha\rho-1}(t+1)^{\alpha\hat\rho-1} dt.
    \]
  \end{lemma}

  Next we prove Proposition \ref{prop:pointoffurthestreach} by expressing  exit probabilities in terms of $\bar\Phi$. In the spirit of the proof of Proposition \ref{prop:pointofclosestreach}, we  apply a linear spatial transformation  to the probability  $\P_x(\tau^{(-u,v)^c}<\tau^{\{0\}})$ and write it in terms of $\bar\Phi$.

  \begin{proof}[Proof of Proposition \ref{prop:pointoffurthestreach}]
    Similar to the derivation of \eqref{eq:minreach_phi} in the proof of Proposition \ref{prop:pointofclosestreach}, for each $x>0$ and $|z|>x$,
    \begin{equation}\label{eq:maxpoint_phi}
      \frac{\P_x(X_{\overline m}\in dz)}{dz} =  \frac{\alpha-1}{2 x^{2-\alpha} |z|^\alpha} \left(|x+z| \bar\Phi'\left(\frac{|z|}{x}\right)-(\alpha-1)x \bar\Phi\left(\frac{|z|}{x}\right)\right).
    \end{equation}
    The result now follows from straight forward computations.
  \end{proof}

  Again we introduce the following lemma from the Appendix of \cite{dereich2015real} (and again, the subtle issue of regularity of $\xi$ for the positive and negative half-lines is being used).

  \begin{lemma}\label{lemma:crossing_a_big}
    Let $T_0^+:=\inf\{t \geq 0: \xi_t >0\}$, then, with the normalisation given in \eqref{eq:killing_ass}, for $i,j=-1,1$ and $y>0$,
    \[
      \bP_{-y,i}(T_0^+=\infty; J_{\varphi(\overline m)} = j)={\bm U}_{i,j}(y).
    \]
  \end{lemma}

  Similar to the derivation in \eqref{eq:u_phi_prime}, we use \eqref{eq:maxpoint_phi} and Lemma \ref{lemma:crossing_a_big} to get that
  \begin{align}\label{eq:u_phi_a_big}
    \bm u_{1,1}(x) &= \frac{\d}{dx} \P_{e^{-x}}(X_{\overline m} \in (e^{-x},1)) \nonumber\\
    &= \frac{\alpha-1}{2}\frac{\d}{dx} \int_{e^{-x}}^1 dz\frac{1}{z^\alpha} e^{(2-\alpha)x}\left((e^{-x}+z) \bar\Phi'\left(\frac{z}{e^{-x}}\right)-(\alpha-1)e^{-x} \bar\Phi\left(\frac{z}{e^{-x}}\right)\right)\nonumber \\
    &= \frac{\alpha-1}{2}\frac{\d}{dx} \int_1^{e^x} du\, \frac{1}{u^\alpha} \left((1+u) \bar\Phi'\left(u\right)-(\alpha-1) \bar\Phi\left(u\right)\right) \nonumber\\
    &= \frac{\alpha-1}{2}  e^{-(\alpha-1)x}\left((1+e^x) \bar\Phi'\left(e^x\right)-(\alpha-1) \bar\Phi\left(e^x\right)\right)\nonumber\\
    &= \frac{\alpha-1}{2}  e^{-(\alpha-1)x} \left((e^x-1)^{\alpha\rho-1}(e^x+1)^{\alpha\hat\rho}-(\alpha-1) \bar\Phi\left(e^x\right)\right),
  \end{align}
  where in the third equality we have used the substitution $u=z/e^{-x}$. Now we will take the Laplace transform of $\bm u_{1,1}$. The transform of the $\bar\Phi$ term is dealt with in the following lemma. The proof follows from integration by parts which we leave out.

  \begin{lemma}\label{lemma:integrate}
    Suppose that $\gamma> \alpha-1$, then
    \[
      \int_0^\infty e^{-\gamma x}\bar\Phi(e^x)dx = \frac{1}{\gamma} \Psi(\gamma-\alpha,\alpha\rho-1,\alpha\hat\rho-1).
    \]
  \end{lemma}

  Next we have
  \[
    \int_0^\infty e^{-(\lambda+\alpha-1)x} (e^x-1)^{\alpha\rho-1}(e^x+1)^{\alpha\hat\rho}dx =\int_0^1 u^{\lambda-1}(1-u)^{\alpha\rho-1}(1+u)^{\alpha\hat\rho}du = \Psi(\lambda-1,\alpha\rho-1,\alpha\hat\rho),
  \]
  where we have used the substitution $u=e^{-x}$. Integrating \eqref{eq:u_phi_a_big} and using the above equation together with Lemma \ref{lemma:integrate} we get that
  \[
    \bm\kappa^{-1}_{1,1}(\lambda)=\frac{\alpha-1}{2}\Psi(\lambda-1,\alpha\rho-1,\alpha\hat\rho) - \frac{(\alpha-1)^2}{2(\lambda+\alpha-1)}\Psi(\lambda-1,\alpha\rho-1,\alpha\hat\rho-1).
  \]
 Similar proofs give $\bm\kappa_{i,j}^{-1}$ for the remaining $i,j$. The given formula for $\hat{\bm\kappa}^{-1}$ follows from Lemma \ref{lemma:K_hat_transform} as well as the straightforward
the matrix algebra 
 \begin{equation}
 \boldsymbol{\Delta}_{\boldsymbol \pi}^{-1} {\bm M} \boldsymbol{\Delta}_{\boldsymbol \pi}\left[
\begin{matrix}
     \frac{\Gamma(1-\alpha\rho) }{\Gamma(\alpha\hat\rho)}& 0 \\
    0 &   \frac{\Gamma(1-\alpha\hat\rho) }{\Gamma(\alpha\rho)} 
          \end{matrix}
      \right] = \left[
\begin{matrix}
     \frac{\Gamma(1-\alpha\rho) }{\Gamma(\alpha\hat\rho)}& 0 \\
    0 &   \frac{\Gamma(1-\alpha\hat\rho) }{\Gamma(\alpha\rho)} 
          \end{matrix}
      \right]{\bm M},
      \label{matrixalgebra}
 \end{equation}
 where ${\bm M}$ is any $2\times 2$ matrix.

  \section{Proof of Theorem \ref{thm:factorisation} for \texorpdfstring{$\alpha=1$}{a=1}}\label{sec:proof_a_1}
  In the case when $\alpha=1$, the process $X$ is a Cauchy process, which has the property that $\limsup_{t\to\infty}|X_t| = \infty$ and $\liminf_{t\to\infty}|X_t|=0$. This means that the MAP $(\xi,J)$ oscillates and hence the global minimum and maximum both do not exist so that the previous methods cannot be used. Instead we focus on a two sided exit problem as an alternative approach. (Note, the method we are about to describe also works for the other cases of $\alpha$, however it is lengthy and we do not obtain the new identities {\it en route} in a straightforward manner as we did in Proposition \ref{prop:pointofclosestreach} and Proposition \ref{prop:pointoffurthestreach}.)

  The following result follows from the compensation formula and the proof of it is identical to the case for L\'evy processes, see \cite[Chapter III Proposition 2]{MR1406564} and \cite[Theorem 5.8]{MR3155252}.

  \begin{lemma}\label{lemma:cauchy_exit_potential}
    Let $(H^+,J^+)$ be the height process of $(\xi,J)$. For any $x >0$ define $T_x:=\inf\{t>0: H^+_t>x\}$, then for any $x>0$ and $i=\pm 1$,
    \[
      \bP_{0,i}( x-H^+_{T_x-}\in du; J^+_{T_x-}=1; J^+_{T_x}=1 ) = {\bm u}_{i,1}(x-u) \Lambda[u,\infty)du,
    \]
    where $\Lambda$ is some $\sigma$--finite measure on $[0,\infty)$.
    \end{lemma}

    Next we will calculate the over and under shoots in Lemma \ref{lemma:cauchy_exit_potential} by using the underlying process $X$. This is done in the following lemma.

    \begin{lemma}\label{lemma:cauchy_exit_prob}
      Let $\tau^+_1:=\inf\{t \geq 0: X_t >1\}$ and $\tau^-_{-1}:=\inf\{t \geq 0: X_t <-1\}$. Then for $x \in (-1,1)$, $u\in[0,(1-x)\vee 1)$ and $y \geq 0$,
        \begin{align*}
          \P_x(1 -&\bar X_{\tau^+_1-} \in du; X_{\tau^+_1}-1 \in dy; X_{\tau^+_1-}>0;\tau^+_1<\tau^-_{-1})= \frac{(1-u+x)^{1/2}}{(1-u-x)^{1/2}} \frac{(u+y)^{3/2}}{(2-u+y)^{1/2}}dudy. 
        \end{align*}
      \end{lemma}
      \begin{proof}
        First \cite[Corollary 3]{kyprianou2014potentials} gives that for $z \in (0,1)$, $u\in[0,1-z)$ and $v \in (u,1]$,
        \begin{align*}
          \P_z&(1 -\bar X_{\tau^+_1-} \in du; 1-X_{\tau^+_1-}\in dv;X_{\tau^+_1}-1 \in dy; \tau^+_1<\tau^-_{0}) \\
          &= \frac{1}{\pi} \frac{z^{1/2}(1-v)^{1/2}}{(1-u-z)^{1/2}(v-u)^{1/2} (1-u)(v+y)^2}dudy,
        \end{align*}
        where $\tau^-_{0}:=\inf\{t \geq 0: X_t <0\}$. We wish to integrate $v$ out of the above equation. To do this, we make the otherwise subtle observation that
        \begin{align*}
          \int_u^1dv \, (1-v)^{1/2} (v-u)^{-1/2} (v+y)^{-2}&= (u+y)^{-2}(1-u)\int_0^1dz \, (1-z)^{1/2} z^{-1/2} \left(1+z\frac{1-u}{u+y}\right)^{-2} \\
          &= (u+y)^{-2}(1-u) \frac{\pi}{2}\, \,_2F_1\left(2,1/2,2; -\frac{1-u}{u+y}\right)\\
          &=\frac{\pi}{2} (u+y)^{-3/2}(1-u) (1+y)^{-1/2},
        \end{align*}
        where in the first equality we have used the substitution $z=(v-u)/(1-u)$. In the second equality we have used \cite[Theorem 2.2.1]{MR1688958}  and the final equality follows from the Euler--transformation \cite[Theorem 2.2.5]{MR1688958}.

        Hence, for $z\in (0,1)$, $u\in[0,1-z)$ and $y \geq 0$,
          \begin{align}\label{eq:double_exit_unshifted}
            \P_z&(1 -\bar X_{\tau^+_1-} \in du; X_{\tau^+_1}-1 \in dy; \tau^+_1<\tau^-_{0}) \nonumber\\
            &= \frac{1}{2} \frac{z^{1/2}}{(1-u-z)^{1/2}(1+y)^{1/2} (u+y)^{3/2}} dudy. 
          \end{align}

          Next we have that for $x\in (-1,1)$, $u \in [0,(1-x)\vee 1)$ and $y \geq 0$,
            \begin{align*}
              &\frac{\P_x(1 -\bar X_{\tau^+_1-} \in du; X_{\tau^+_1}-1 \in dy; \bar X_{\tau^+_1-}>-\underline X_{\tau^+_1-};\tau^+_1<\tau^-_{-1})}{du\, dy}\\
              &= \frac{\partial}{\partial v}\frac{\partial}{\partial y}\P_x(1 -\bar X_{\tau^+_1-} \leq v; X_{\tau^+_1}-1 \leq y; \tau^+_1<\tau^-_{u-1})|_{v=u}\\
              &= \frac{\partial}{\partial v}\frac{\partial}{\partial y}\P_{\frac{x+1-u}{2-u}}\left(1-\bar X_{\tau^+_1-} \leq \frac{v}{2-u}; X_{\tau^+_1}-1 \leq \frac{y}{2-u}; \tau^+_1<\tau^-_{0}\right)|_{v=u}\\
              & = \frac{1}{2}(2-u)^{-2} \frac{\left(\frac{x+1-u}{2-u}\right)^{1/2}}{\left(1-\frac{u}{2-u}-\frac{x+1-u}{2-u}\right)^{1/2}\left(1+\frac{y}{2-u}\right)^{1/2} \left(\frac{u+y}{2-u}\right)^{3/2}}\\
              &= \frac{(1-u+x)^{1/2}}{(1-u-x)^{1/2}} \frac{1}{(2-u+y)^{1/2}(u+y)^{3/2}},
            \end{align*}
            where in the first equality we have used that the event $\{\bar X_{\tau^+_1-}>-\underline X_{\tau^+_1-}, 1 -\bar X_{\tau^+_1-} \in du\}$ constrains $\underline{X}$ and thus  is equivalent to $\{ \tau^+_1<\tau^-_{u-1}, 1 -\bar X_{\tau^+_1-} \in du\}$. In the second equality we have used the scaling property of $X$ and in the third equality we have used \eqref{eq:double_exit_unshifted}.
          \end{proof}

          Notice now that for each $x \geq 0$ and $j = \pm 1$,
          \begin{align}\label{eq:Lambda_find}
            &\frac{\partial}{\partial u}\frac{\partial}{\partial y}\bP_{0,j}(x-H^+_{T_x-}\leq u; H^+_{T_x}-x\leq y; J^+_{T_x-}=1; J^+_{T_x}=1 ) \nonumber\\
            &=\frac{\partial}{\partial u}\frac{\partial}{\partial y}\P_{j}(\bar X_{\tau^+_{e^x}-} \geq e^{x-u};X_{\tau^+_{e^x}} \leq e^{y+x}; X_{\tau^+_{e^x}-}>0;\tau^+_{e^{x}}<\tau^-_{-e^{x}})\nonumber\\
            &=\frac{\partial}{\partial u}\frac{\partial}{\partial y}\P_{je^{-x}}(\bar X_{\tau^+_{1}-} \geq e^{-u};X_{\tau^+_{1}} \leq e^y; X_{\tau^+_{1}-}>0;\tau^+_{1}<\tau^-_{-1})\nonumber\\
            &=e^{y-u}\frac{(e^{-u}+je^{-x})^{1/2}}{(e^{-u}-je^{-x})^{1/2}} \frac{1}{(e^y+e^{-u})^{1/2}(e^y-e^{-u})^{3/2}},
          \end{align}
                   where in the second equality we have used the scaling property of $X$ and in the final equality we applied Lemma \ref{lemma:cauchy_exit_prob}. The above equation together with Lemma \ref{lemma:cauchy_exit_potential} gives that for $x \geq 0$,
          \begin{equation}\label{eq:cauchy_sim_eq1}
            \frac{\bm{u}_{1,1}(x-u)}{\bm{u}_{-1,1}(x-u)}= \frac{\bP_{0,1}( x-H^+_{T_x-}\in du; J^+_{T_x-}=1; J^+_{T_x}=1)/du}{\bP_{0,-1}( x-H^+_{T_x-}\in du; J^+_{T_x-}=1; J^+_{T_x}=1 )/du} = \frac{1+e^{-(x-u)}}{1-e^{-(x-u)}}.
          \end{equation}

          Next we claim that for any $x\geq 0$,
          \begin{equation}\label{eq:cauchy_pot_sum}
            \sum_{i=\pm 1}\bm{u}_{1,i}(x) = (1-e^{-x})^{-1/2}(1+e^{-x})^{1/2} + (1-e^{-x})^{1/2}(1+e^{-x})^{-1/2},
          \end{equation}
          which also fixes the normalisation of local time (not necessarily as in \eqref{eq:killing_ass}). Again we remark that this is not a concern on account of the fact that Theorem \ref{thm:factorisation} is stated up to post-multiplication by a strictly positive diagonal matrix. 
          This follows from existing literature on the Lamperti transform of the Cauchy process and we briefly describe how to verify it. It is known (thanks to scaling of $X$ and symmetry)  that $(|X_t|:t \geq 0)$ is a positive self-similar Markov process with index $\alpha$. As such, it can can be expressed in the form 
          $
          |X_t| = \exp\{\chi_{\beta_t}\},
          $
          for $ t\leq \tau^{\{0\}}$,
          where $\beta_t = \inf\{s>0 : \int_0^s \exp\{\alpha \chi_u\}du >t\}$, see for example \cite[Chapter 13]{MR3155252}. The sum on the left-hand side of \eqref{eq:cauchy_pot_sum} is precisely the potential of the ascending ladder height process of the L\'evy process $\chi$. We  can verify that the potential of the ascending ladder height process of $\chi$ has the form given by the right-hand side of \eqref{eq:cauchy_pot_sum} as follows. Laplace exponent of the ascending ladder height process of $\chi$ is given in \cite[Remark 2]{MR2797981}. Specifically, it takes the form 
          $\kappa_\chi(\lambda): = \Gamma((\lambda +1 )/2)/\Gamma(\lambda/2)$, $\lambda \geq 0$. Then the identity in \eqref{eq:cauchy_pot_sum} can be verified by checking that, up to a multiplicative constant,  its Laplace transform agrees with $1/\kappa_\chi(\lambda)$, $\lambda\geq 0$.

          Now we can finish the proof. Notice first that the Cauchy process is symmetric, thus $\bm{u}_{i,j}=\bm{u}_{-i,-j}$ for each $i,j\in\{1,-1\}$. Thus from \eqref{eq:cauchy_pot_sum} we get
          \begin{equation}\label{eq:cauchy_sim_eq2}
            \sum_{i=\pm 1}\bm{u}_{i,1}(x) = (1-e^{-x})^{-1/2}(1+e^{-x})^{1/2} + (1-e^{-x})^{1/2}(1+e^{-x})^{-1/2}.
          \end{equation}
          Solving the simultaneous equations \eqref{eq:cauchy_sim_eq1} and \eqref{eq:cauchy_sim_eq2} together with the fact $\bm{u}_{i,j}=\bm{u}_{-i,-j}$ gives the result for $\bm{u}$. To obtain $\hat{\bm{u}}$ we note that the reciprocal process $\widetilde{X}: = 1/X_{\theta_t}$, $t \geq 0$ has the law of a Cauchy process, where $\theta_t = \inf\{s>0: \int_0^{s}|X_u|^{-2}du >t\}$ (see \cite[Theorem 1]{MR2256481}). Theorem 4 in \cite{kyprianou2015deep} also shows that $\widetilde{X}$ has an underlying MAP which is the dual of the MAP underlying $X$. It therefore follows that
          $\hat{\bm{u}}=\bm u$. This finishes the proof.

          \begin{rem}
            Using the form of $\bm u$ and \eqref{eq:Lambda_find}, we also get the jump measure $\Lambda$ appearing in Lemma \ref{lemma:cauchy_exit_potential} as
            \[
              \Lambda(d y)= \frac{e^{y}}{(e^y+1)^{1/2}(e^y-1)^{3/2}}\d y.
            \]
            Up to a multiplicative constant, this can also be seen in \cite[equation (14)]{kyprianou2015deep}.
          \end{rem}

          \section{Proof of Theorem \ref{thm:reflecting}}\label{sec:reflecting}

          Recall that $(R,M)$ is a Markov process. Since $R$ takes values on $[-1,1]$ and is recurrent, it must have a limiting distribution which does not depend on its initial position.  For $x\in[-1,1]$ and $j = \pm1$, when it exists, define
          \[
            \mu_j(dy): = \lim_{t\to\infty}\mathbb{P}_x(|R_t|\in dy; \, {\rm sgn}(R_t) = j)\qquad y \in [0,1].
          \]
          Notice that the stationary distribution $\mu$ is given by $\mu(A) = \mu_1(A\cap[0,1])+ \mu_{-1}(-A\cap[0,1])$ (here we are pre-emptively assuming that each of the two measures on the right-hand side are absolutely continuous with respect to Lebesgue measure and so there is no `double counting' at zero) and hence it suffices to establish an identify for $\mu_j$.

          For $i,j=\pm1$,
          \[
            \bE_{0,i}\left[ e^{-\beta (\overline{\xi}_{\mathbbm{e}_q} - \xi_{\mathbbm{e}_q})} ; J_{\mathbbm{e}_q}=j\right]= \sum_{k= \pm1}\bE_{0,i}\left[ e^{-\beta (\overline{\xi}_{\mathbbm{e}_q} - \xi_{\mathbbm{e}_q})} ; J_{\overline{m}_{\mathbbm{e}_q}}=k, \,J_{\mathbbm{e}_q}=j\right],
          \]
          where $\mathbbm{e}_q$ is an independent and exponentially distributed random variable with rate $q$ and $\overline{m}_{\mathbbm{e}_q} $ is the unique time at which $\xi$ obtains its maximum on the time interval $[0,\mathbbm{e}_q]$. Appealing to the computations in the Appendix of \cite{dereich2015real}, specifically equation (22) and Theorem 23, we can develop the right-hand side above using duality, so that
          \begin{eqnarray*}
            \int_0^1 y^{\beta}\tilde\mu_j(d y)&:=&\lim_{q\downarrow0}\bE_{0,i}\left[ e^{-\beta (\overline{\xi}_{\mathbbm{e}_q} - \xi_{\mathbbm{e}_q})} ; J_{\mathbbm{e}_q}=j\right]\\
            &&=\lim_{q\downarrow0}\sum_{k=\pm1}\bE_{0,i}\left[ e^{-\beta (\overline{\xi}_{\mathbbm{e}_q} - \xi_{\mathbbm{e}_q})} ;
            J_{\overline{m}_{\mathbbm{e}_q}}=k, \,
          J_{\mathbbm{e}_q}=j\right]\\
          &&=\lim_{q\downarrow0} \sum_{k= \pm1}\bP_{0,i}\left( J_{\overline{m}_{\mathbbm{e}_q}}=k\right)\hat{\bE}_{0,j}\left[ e^{-\beta \overline{\xi}_{\mathbbm{e}_q} } ; J_{\overline{m}_{\mathbbm{e}_q}}=k\right]\frac{\pi_j}{\pi_k}\\
          &&={\pi_j}\sum_{k=\pm1}[\hat{\bm\kappa}(\beta)^{-1}]_{j,k}c_k,
        \end{eqnarray*}
        for some strictly positive constants $c_{\pm1}$,
           where in the first equality we have used the Lamperti--Kiu transform and in the third equality, we have split the process at the  maximum and used that, on the event $\{J_{\overline{m}_{\mathbbm{e}_q}} = j, J_{\mathbbm{e}_q} = k\}$, the pair $( \overline{\xi}_{\mathbbm{e}_q}-\xi_{\mathbbm{e}_q} , \mathbbm{e}_q - \overline{m}_{\mathbbm{e}_q} )$ is equal in law to the pair
        $( \overline{\hat{\xi}}_{\mathbbm{e}_q},  {\overline{\hat{m}}}_{\mathbbm{e}_q} )$ on $\{\hat{J}_0 = k, \hat{J}_{\overline{\hat{m}}_{\mathbbm{e}_q}}  =j\}$, where $\{(\hat\xi_s, \hat{J}_s): s\leq t\}: = \{ (\xi_{(t-s)-} -\xi_t, J_{(t-s)-}): s\leq t\}$, $t\geq 0$, is equal in law to the dual of $\xi$, $\overline{\hat{\xi}}_t = \sup_{s\leq t}\hat{\xi}_s$ and $\overline{\hat{m}} = \sup\{s\leq t: \overline{\hat\xi}_s = \hat{\xi}_t\}$.
        Note, we have also used the fact that, $\overline{m}_{\mathbbm{e}_q}$ converges to $+\infty$ almost surely as $q\to\infty$ on account of the fact that $\limsup_{t\to\infty}|X_t|=\infty$.

Since $[\hat{\bm\kappa}(\lambda)^{-1}]_{j,k}$ is the Laplace transform of $\hat{\bm u}_{j,k}$, it now follows that,
\[
\left.\frac{\d\tilde\mu_j(y)}{{d y}}\right|_{y = {\rm e}^{-x}}={\pi_j} \sum_{k=\pm1}\hat{\bm u}_{j,k}(x)c_k, \qquad x\geq 0.
\]
Said another way,
\[
\tilde\mu_j({d y})  = \frac{\pi_j}{y}\sum_{k=\pm1}\hat{\bm u}_{j,k}(-\log y)c_k\d y, \qquad y\in[0,1].
\]
The constants $c_k$, $k=\pm1$, can be found by noting that, for $j  = \pm1$, $\mu_j([0,1]) = \pi_j$ and hence, for $j =\pm1$,
\begin{equation}
c_1\left(\int_0^\infty \hat{\bm u}_{j,1}(x) \d x\right)  + c_{-1}\left(\int_0^\infty \hat{\bm u}_{j,-1}(x) \d x\right)= 1.
\label{c1c-1}
\end{equation}
Using~\cite{hyper} and Theorem \ref{cor:potentials} (i),
\begin{align*}
&\int_0^\infty [\hat{\bm u}_{1,1}(x) - \hat{\bm u}_{-1,1}(x)]\d x \\
&= \frac{\Gamma(1-\alpha\rho)}{\Gamma(\alpha\hat\rho)} \int_0^1 u^{-\alpha} (1-u)^{\alpha\hat\rho -1}(1+u)^{\alpha\rho}
{\d u}\\
&\hspace{1cm}- 
\frac{\Gamma(1-\alpha\hat\rho)}{\Gamma(\alpha\rho)}\int_0^1 u^{-\alpha} (1-u)^{\alpha\rho}(1+u)^{\alpha\hat\rho-1}du
\\
&=  \frac{\Gamma(1-\alpha\rho)}{\Gamma(\alpha\hat\rho)}B(1-\alpha, \alpha\hat\rho)
\,\,_{2}F_1(-\alpha\rho, 1-\alpha, 1-\alpha\rho; -1)\\
&\hspace{1cm}-
\frac{\Gamma(1-\alpha\hat\rho)}{\Gamma(\alpha\rho)}B(1-\alpha, \alpha\rho + 1)
\,\, _{2}F_1(1-\alpha\hat\rho, 1-\alpha, 2-\alpha\hat\rho; -1)\\
&=\Gamma(1-\alpha\rho)\Gamma(1-\alpha\hat\rho).
\end{align*}
Now subtracting \eqref{c1c-1} in the case $j = -1$ from the case $j =1$, it appears that 
\[
\Gamma(1-\alpha\rho)\Gamma(1-\alpha\hat\rho)(c_1 - c_{-1}) = 0,
\]
which is to say, $c_1 = c_{-1}.$ 

In order to evaluate either of these constants, we appeal to  the definition of the Beta function to compute
\begin{align*}
&\int_0^\infty [\hat{\bm u}_{1,1}(x) +\hat{\bm u}_{1,-1}(x)]\d x \\
&=\frac{\Gamma(1-\alpha\rho)}{\Gamma(\alpha\hat\rho)} \int_0^1
u^{-\alpha}(1-u)^{\alpha\hat\rho -1}(1+u)^{\alpha\rho}
+
u^{-\alpha}(1-u)^{\alpha\hat\rho}(1+u)^{\alpha\rho -1}
\d u\\
&=2\frac{\Gamma(1-\alpha\rho)}{\Gamma(\alpha\hat\rho)} \int_0^1
u^{-\alpha}(1-u)^{\alpha\hat\rho -1}(1+u)^{\alpha\rho-1}
\d u\\
%
&=2^{\alpha}\frac{\Gamma(1-\alpha\rho)}{\Gamma(\alpha\hat\rho)} \int_0^1 v^{\alpha\hat\rho-1}(1-v)^{-\alpha}\d v \\
&=2^{\alpha}\Gamma(1-\alpha), 
\end{align*}
where in the third equality, we have made the substitution $v = (1-u)/(1+u)$. It now follows from \eqref{c1c-1} that 
\[
c_1= c_{-1} = \frac{1}{2^{\alpha}\Gamma(1-\alpha)}
\]
and hence e.g. on $y\in[0,1]$,
\begin{align*}
\tilde{\mu}(dy)& =\frac{\sin(\pi\alpha\rho)\Gamma(1-\alpha\rho)}{2^{\alpha}\Gamma(\alpha\hat\rho)\Gamma(1-\alpha)[\sin(\pi\alpha\rho) + \sin(\pi\alpha\hat\rho)]}
\left\{y^{-\alpha} (1-y)^{\alpha\hat\rho -1}(1+y)^{\alpha\rho} +y^{-\alpha}(1-y)^{\alpha\hat\rho}(1+y)^{\alpha\rho-1} \right\}\\
&=\frac{2^{-\alpha}\pi}{\Gamma(\alpha\rho)\Gamma(\alpha\hat\rho)\Gamma(1-\alpha)[\sin(\pi\alpha\rho) + \sin(\pi\alpha\hat\rho)]}
\left\{y^{-\alpha} (1-y)^{\alpha\hat\rho -1}(1+y)^{\alpha\rho} +y^{-\alpha}(1-y)^{\alpha\hat\rho}(1+y)^{\alpha\rho-1} \right\}
\end{align*}
The proof is completed by taking account of the the time change in the representation (2) in the limit (see for example the discussion at the bottom of p240 of \cite{Walsh} and references therein) and noting that, up to normalisation by a constant, $K$,
\[
\mu_j(dy) = Ky^\alpha \tilde\mu_j(dy), 
\]
for $j = \pm1$.
Note that 
\begin{align*}
1 &=\int_{-1}^1\mu(dy)
\\&= K\frac{2^{1-\alpha}\pi}{\Gamma(\alpha\rho)\Gamma(\alpha\hat\rho)\Gamma(1-\alpha)[\sin(\pi\alpha\rho) + \sin(\pi\alpha\hat\rho)]}\\
&\times
\left\{
\int_0^1  (1-y)^{\alpha\hat\rho -1}(1+y)^{\alpha\rho-1} dy
+\int_0^1  (1-y)^{\alpha\rho -1}(1+y)^{\alpha\hat\rho-1} dy
\right\}.
\end{align*}
Appealing to the second hypergeometric identity in \cite{hyper}, the curly brackets is equal to 
\begin{align*}
&\frac{1}{\alpha\hat\rho}\, {_{2}F_1}(1-\alpha\rho, 1, 1+\alpha\hat\rho;-1) + \frac{1}{\alpha\rho}\, {_{2}F_1}(1-\alpha\hat\rho,1,1+\alpha\rho;-1)\\
&=\frac{1}{\alpha\hat\rho} \frac{1}{\alpha\rho}\left(
\alpha\rho\,{_{2}F_1}(1-\alpha\rho, 1, 1+\alpha\hat\rho;-1) + \alpha\hat\rho\,{_{2}F_1}(1-\alpha\hat\rho,1,1+\alpha\rho;-1)\right)\\
&=2^{\alpha-1} \frac{\Gamma(\alpha\rho)\Gamma(\alpha\hat\rho)}{\Gamma(\alpha)}
\end{align*}
and hence 
\[
K = \frac{[\sin(\pi\alpha\rho) + \sin(\pi\alpha\hat\rho)]}{\sin(\alpha\pi)}.
\]
In conclusion, we have that 
\begin{align*}
   \frac{d\mu(y)}{dy} =2^{-\alpha} \frac{\Gamma(\alpha)}{\Gamma(\alpha\rho)\Gamma(\alpha\hat\rho)}\begin{cases}
  (1-y)^{\alpha\hat\rho -1}(1+y)^{\alpha\rho} +(1-y)^{\alpha\hat\rho}(1+y)^{\alpha\rho-1} & \text{ if }y \in [0,1]\\
    &\\
 (1-|y|)^{\alpha\rho}(1+|y|)^{\alpha\hat\rho -1} +(1-|y|)^{\alpha\rho -1}(1+|y|)^{\alpha\hat\rho}
   & \text{ if } y \in [-1,0) .
  \end{cases}
\end{align*}

as required. \hfill$\square$

        \section*{Acknowledgements}
        We would like to thank the two anonymous referees for their valuable feedback. We would also like to thank Mateusz Kwa\'snicki for his insightful remarks.
        AEK and B\c{S} acknowledge support from EPSRC grant number EP/L002442/1. AEK and VR acknowledge support from EPSRC grant number EP/M001784/1. VR acknowledges support from CONACyT grant number 234644. This work was undertaken whilst VR was on sabbatical at the University of Bath, he gratefully acknowledges the kind hospitality and the financial support of the Department and the University.

        \printbibliography

        \end{document}